\documentclass[11pt]{article}

\usepackage{paralist}
\usepackage[margin=3.0cm]{geometry} 
\usepackage{graphicx} 

\usepackage{amssymb, amsfonts, amsmath, amscd, amsthm} 
\usepackage{mathtools} 
\usepackage{titlesec} 
\usepackage{csquotes} 
\usepackage[shortlabels]{enumitem} 
\usepackage{tikz-cd} 
\usepackage{wrapfig} 
\usepackage{epigraph} 
\usepackage{hyperref} 
\usepackage{todonotes} 
\usepackage{comment}
\usepackage{mathrsfs} 
\usepackage{todonotes}
\usepackage{soul}
\usepackage[misc]{ifsym}

\hypersetup{
     colorlinks   = true,
     citecolor    = green
}

\numberwithin{equation}{section}

\titleformat*{\section}{\Large \scshape\center}
\titleformat*{\subsection}{\fontsize{14}{14} \sffamily}

\theoremstyle{plain}
\newtheorem{theorem}{Theorem}[section]
\newtheorem*{theorem*}{Theorem}

\newtheorem{lemma}[theorem]{Lemma}
\newtheorem{proposition}[theorem]{Proposition}
\newtheorem{corollary}[theorem]{Corollary}

\theoremstyle{definition}

\newtheorem*{definition*}{Definition}
\newtheorem{example}[theorem]{Example}

\theoremstyle{remark}
\newtheorem{remark}{Remark}

\DeclareMathOperator{\tr}{tr}
\DeclareMathOperator{\re}{Re}
\DeclareMathOperator{\im}{Im}

\begin{document}
\pagenumbering{gobble}
\title{Commutative $G$-invariant Toeplitz C$^\ast$-algebras on the Fock space and their Gelfand theory through Quantum Harmonic Analysis}
\author{Robert Fulsche and Miguel Angel Rodriguez Rodriguez}
\date{}

\maketitle
\begin{abstract}
We discuss the notion of spectral synthesis for the setting of Quantum Harmonic Analysis. Using these concepts, we study subalgebras of the full Toeplitz algebra with certain invariant symbols and their commutators. In particular, we find a new class of commutative Toeplitz $C^\ast$-algebras on the Fock space. In the end, we investigate the Gelfand theory of those commutative C$^\ast$-algebras.


\end{abstract}

\pagenumbering{arabic}

\section{Introduction}
In recent years, there has been a steady interest in studying commutative $C^\ast$ or Banach algebras generated by Toeplitz operators, see e.g. \cite{Bauer_Rodriguez2022, Dawson_Olafsson_Quiroga2020, Vasilevski2008}. Usually, this is related to the symbols of the Toeplitz operators obeying certain symmetries or invariances. On the Fock space, focus has been on radial symbols \cite{Dewage_Olafsson2022, Esmeral_Maximenko2016} or Lagrangian symbols \cite{Esmeral_Vasilevski2016}. Recently, the class of Lagrangian symbols was extended to the class of $k$-isotropic symbols and the center of the resulting Toeplitz algebra was investigated \cite{Hernandez}. The case of Lagrangian- or $k$-isotropic symbols results in operator algebras, which are (in a sense to be made precise) shift-invariant closed subspaces of the full Toeplitz algebra. Such spaces lie at the heart of a different method for investigating operators on the Fock space, namely Quantum Harmonic Analysis and, as a special part of it, Correspondence Theory. Quantum Harmonic Analysis originates from the 80s \cite{werner84} and has been ignored by mathematicians for a long time. In recent years, the interest in this topic grew significantly, see e.g. \cite{Berge_Berge_Luef_Skrettingland2022, Fulsche2020, keyl_kiukas_werner16, Luef_Skrettingland2021}.

As has been demonstrated in \cite{Fulsche2020, Fulsche2022}, Correspondence Theory is a very powerful tool to study operators on the Fock space, as it allows for very short proofs for deep theorems, drawing its elegance from Werner's formalism of Quantum Harmonic Analysis \cite{werner84}. Indeed, results such as the compactness characterization of Bauer and Isralowitz \cite{Bauer_Isralowitz2012} or the results of Xia \cite{Xia2015} reduce to a proof of no more than five lines.

This paper is supposed to demonstrate once again the usefulness of Quantum Harmonic Analysis in the study of Toeplitz operators on Fock space. In Section 2 of this paper, we briefly describe some aspects of Quantum Harmonic Analysis, most notably the QHA analogue of the spectral synthesis problem. In essence, almost everything from this section has already been contained in Werner's initial work on Quantum Harmonic Analysis \cite{werner84}. Since Werner's paper is written quite densely and the part on Correspondence Theory in weak$^\ast$ topology is very brief in \cite{werner84}, we take the freedom to discuss some results on this matter in detail. In particular, we prove that the problem of ``quantum spectral synthesis'' is equivalent to the problem of classical spectral synthesis. In Section 3, we apply these results to study $C^\ast$-algebras of the form 
\begin{align*}
\mathcal A_G = \{ A \in \mathcal L(F^2): ~W_z A W_{-z} = A \text{ for every } z \in G\},
\end{align*}
where $F^2 = F^2(\mathbb C^d)$ is the Segal-Bargmann-Fock space, $W_z$ are the Weyl operators on $F^2$ and $G$ is a closed additive subgroup of $\mathbb C^d$. In particular, we compute the commutant of $\mathcal A_G$ and describe its intersection with the full Toeplitz algebra. Since the intersection with the Toeplitz algebra is again a Toeplitz algebra, this yields a description of the centers of certain invariant Toeplitz algebras. Having the center described, we obtain a characterization when such algebras are commutative. This recovers the result from \cite{Esmeral_Vasilevski2016} for the Lagrangian-invariant algebras, but we also find some new commutative Toeplitz $C^\ast$-algebras, which seemingly were not known before in the literature. In particular, we single out a nice class of model spaces for this class of commutative Toeplitz algebras. As it is customary, one has to discuss the Gelfand theory upon finding a new class of commutative operator algebras. This is done in Section 4.

\section{Spectral Synthesis for Quantum Harmonic Analysis}

\subsection{A brief recap of the problem of spectral synthesis}
Recall that the closed subspaces of $L^1(\mathbb R^{2d})$ are in a one-one correspondence to the subspaces of $L^\infty(\mathbb R^{2d})$, closed in the weak$^\ast$ topology, as follows: Given a subspace $X_0 \subset L^1(\mathbb R^{2d})$, set
\begin{align*}
    X_0^\perp = \{ f \in L^\infty(\mathbb R^{2d}): ~\langle f, g\rangle = 0 \text{ for every } g \in X_0\}.
\end{align*}
Here,
\begin{align*}
    \langle f, g\rangle = \int_{\mathbb R^{2d}}f(x) g(x)~dx.
\end{align*}
Further, for a subspace $Y_0 \subset L^\infty(\mathbb R^{2d})$ we set
\begin{align*}
    Y_0^\perp = \{ g \in L^1(\mathbb R^{2d}): \langle f, g\rangle = 0 \text{ for every } f \in Y_0\}.
\end{align*}
Then, $X_0^{\perp \perp}$ is the norm closure of $X_0$ and $Y_0^{\perp \perp}$ is the closure of $Y_0$ in weak$^\ast$ topology (with respect to the predual $L^1(\mathbb R^{2d})$ of $L^\infty(\mathbb R^{2d})$). Indeed, as a simple consequence of the Hahn-Banach theorem, this induces a one-one correspondence between closed subspaces of $L^1(\mathbb R^{2d})$ and weak$^\ast$ closed subspaces of $L^\infty(\mathbb R^{2d})$. Further, $X_0$ is closed under the shifts $\alpha_y(f)$ defined by $\alpha_y(f)(x) = f(x-y)$ if and only if $X_0^\perp$ is closed under these shifts. Hence, there is a one-one correspondence between closed shift-invariant subspaces of $L^1(\mathbb R^{2d})$ and $L^\infty(\mathbb R^{2d})$.

The symplectic Fourier transform of $f \in L^1(\mathbb R^{2d})$ is defined as
\begin{align}\label{eq:symplecticfourier}
    \mathcal F_\sigma(f)(\xi) = c \int_{\mathbb R^{2d}} e^{-i\sigma(\xi, z)}f(z)~dz,
\end{align}
where $\sigma$ is a symplectic form on $\mathbb R^{2d}$ and the constant $c$, depending on the symplectic form $\sigma$, is chosen accordingly such that $\mathcal F_\sigma^2 = \operatorname{Id}$. The problem of spectral synthesis is usually formulated with the standard Fourier transform $\mathcal F(f)(\xi) = \int_{\mathbb R^{2d}}e^{-i\xi \cdot z} f(z)~dz$, but for our purposes it will be more suitable to use the symplectic analogue. Indeed, nothing significant changes on the classical side. Given a function $f \in L^1(\mathbb R^{2d})$, we set
\begin{align*}
    Z(f) = \{ \xi \in \mathbb R^{2d}: ~\mathcal F_\sigma(f)(\xi) = 0\}.
\end{align*}
Since $\mathcal F_\sigma(f)$ is continuous, $Z(f)$ is a closed subset of $\mathbb R^{2d}$. Further, given a closed and $\alpha$-invariant subspace $X_0 \subset L^1(\mathbb R^{2d})$, we set
\begin{align*}
    Z(X_0) = \bigcap_{f \in X_0} Z(f).
\end{align*}
This is of course again a closed subset of $\mathbb R^{2d}$. Given any closed subset $I \subset \mathbb R^{2d}$, one can set
\begin{align*}
    X_{I, 0} = \{ f \in L^1(\mathbb R^{2d}): Z(f) \supseteq I\}.
\end{align*}
One can show that this is indeed an $\alpha$-invariant and closed subspace of $L^1(\mathbb R^{2d})$ with $Z(X_{I,0}) = I$. The problem of spectral synthesis asks now: Given any closed subset $I \subset \mathbb R^{2d}$ and $X_0 \subset L^1(\mathbb R^{2d})$ a closed, $\alpha$-invariant subspace with $Z(X_0) = I$, does it follow that $X_0 = X_{I,0}$? A set $I$ satisfying this is called a \emph{set of spectral synthesis}. Indeed, not every closed subset is a set of spectral synthesis, and trying to understand such sets has been an important part of research in harmonic analysis (see the accounts in \cite{Benedetto1975, Hewitt_Ross2, reiter71} for some details on the matter). The first example of a set of spectral synthesis is $\emptyset$, which is the result of Wiener's approximation theorem.

Indeed, there is a dual formulation of the spectral synthesis problem: For a closed (in weak$^\ast$ topology) $\alpha$-invariant subspace $Y_0 \subset L^\infty(\mathbb R^{2d})$ we let
\begin{align*}
    \Sigma(Y_0) = \{ \xi \in \mathbb R^{2d}: ~e^{i\sigma(\xi, \cdot)} \in Y_0\}
\end{align*}
denote the spectrum of $Y_0$. Since $\xi \mapsto e^{i\sigma(\xi, \cdot)}$ is continuous in weak$^\ast$ topology (this follows from the dominated convergence theorem), $\Sigma(Y_0)$ is always closed. Given a subset $I \subset \mathbb R^{2d}$, we denote by $Y_{I,0}$:
\begin{align*}
    Y_{I,0} = \overline{\operatorname{span}} \{ e^{i\sigma(\xi, \cdot)}: ~\xi \in I\} = X_{I, 0}^\perp,
\end{align*}
where the closure is taken in weak$^\ast$ topology. Then, we clearly have $Y_{\Sigma(Y_0), 0} \subset Y_0$ for every closed, $\alpha$-invariant subspace $Y_0$ of $L^\infty(\mathbb R^{2d})$. As a matter of fact, it is:
\begin{align}
    Z(X_0) = \Sigma(X_0^\perp).
\end{align}
Hence, we obtain: A closed subset $I \subset \mathbb R^{2d}$ is a set of spectral synthesis if and only if for $Y_0 \subset L^\infty(\mathbb R^{2d})$ closed and $\alpha$-invariant with $\Sigma(Y_0) = I$ it necessarily holds true that $Y_0 = Y_{I,0}$.

\subsection{Spectral Synthesis for Quantum Harmonic Analysis}
Before starting our discussion, we want to emphasize that spectral synthesis for operators has of course already been studied in close detail. We can recommend the nice article \cite{Arveson1982} of Arveson on the matter, which also contains most results of the present section as special cases. Seemingly unaware of Werner's work, Arveson  also used arguments similar to those in \cite{werner84} in his book \cite{Arveson2003}, compare e.g.\ Section 7.2 in that book with the proof of Theorem \ref{thm:specsynth_qspecsynth} below. Nevertheless, we think that a special treatment of our situation is worth being elaborated, as the framework of operator convolution from Quantum Harmonic Analysis allows for a very elegant treatment of the matter. We start this section by recalling the basic notions from Quantum Harmonic Analysis, cf.\ \cite{werner84} for the original source of these ideas and \cite{Fulsche2020} for the adapted formulations for the Fock space representation. As a phase space, we will always consider $(\mathbb R^{2d}, \sigma) \cong (\mathbb C^d, \sigma)$ with the symplectic form $\sigma(z, w) = 2\im(z \cdot \overline{w})$ for $z, w \in \mathbb C^d$. Since this symplectic form is two times the standard symplectic form on $\mathbb R^{2d}$, the constant $c$ in the symplectic Fourier transform \eqref{eq:symplecticfourier} should be chosen as $c = (\pi)^{-d}$. We will consider operators on the Fock space $F^2(\mathbb C^d)$, which is the closed subspace of $L^2(\mathbb C^d, \mu)$ consisting of entire functions, where $\mu$ is the Gaussian measure
\begin{align*}
    d\mu(z) = \frac{1}{\pi^d} e^{-|z|^2}~dz.
\end{align*}
$F^2(\mathbb C^d)$ is well-known to be a reproducing kernel Hilbert space with kernel functions
\begin{align*}
    K_z(w) = e^{w\cdot \overline{z}}, \quad z, w \in \mathbb C^d.
\end{align*}
The normalized reproducing kernels are $k_z(w) = K_z(w)/\| K_z\| = e^{w\cdot \overline{z} - \frac{|z|^2}{2}}$.

On $F^2(\mathbb C^d)$, the CCR relations of our symplectic space are implemented by the Weyl operators
\begin{align*}
    W_z f(w) = k_z(w) f(w-z),
\end{align*}
i.e.\ the Weyl operators form a family of unitary operators which act irreducibly on $F^2(\mathbb C^d)$, depend continuously on $z$ in strong operator topology and satisfy
\begin{align*}
    W_z W_w = e^{-\frac{i}{2}\sigma(z, w)}W_{z+w}, \quad W_z^{-1} = W_{-z}, \quad z, w \in \mathbb C^d.
\end{align*}
Everything we will discuss in this section can be carried over to every irreducible representation of the CCR relations of a finite dimensional symplectic space. Since we have the applications to Toeplitz operators in mind, we only write this explicit formulation out. 

We now turn towards the description of the quantum spectral synthesis problem. Nevertheless, we want to emphasize that already Werner gave a short discussion on quantum spectral synthesis in his original paper, cf.\ \cite[Corollary 4.4]{werner84}. Since his paper is rather densely written, we decided to give a more detailed discussion of the matter, based on which we can then build our discussion on commutative Toeplitz algebras.

Essentially the same structure as described for the function spaces above holds true for the operator side: Given a closed subspace $X_1 \subset \mathcal T^1(F^2(\mathbb C^d))$, where $\mathcal T^1(F^2(\mathbb C^d))$ denotes the ideal of trace class operators, we let
\begin{align*}
    X_1^\perp = \{ A \in \mathcal L(\mathcal F^2(\mathbb C^d)): ~\langle A, B\rangle = 0 \text{ for all } B \in X_1\}.
\end{align*}
Here, it is
\begin{align*}
    \langle A, B\rangle = \tr(AB),
\end{align*}
where this pairing of course induces the duality $\mathcal T^1(F^2(\mathbb C^d))' \cong \mathcal L(F^2(\mathbb C^d))$. Additionally, for a closed subspace $Y_1 \subset \mathcal L(F^2(\mathbb C^d))$ (closed w.r.t.\ the weak$^\ast$ topology), we let
\begin{align*}
    Y_1^\perp = \{ B \in \mathcal T^1(F^2(\mathbb C^d)): ~\langle A, B\rangle = 0 \text{ for all } A \in Y_1\}.
\end{align*}
Similarly to the shift of functions, the shift of an operator is defined by
\begin{align*}
    \alpha_z(A) = W_z A W_{-z}, \quad z \in \mathbb C^d, ~A \in \mathcal L(F^2(\mathbb C^d)).
\end{align*}
A subspace of $\mathcal L(F^2(\mathbb C^d))$ or of $\mathcal T^1(F^2(\mathbb C^d))$ is said to be $\alpha$-invariant if it is invariant under $\alpha_z$ for every $z \in \mathbb C^d$. Again by the Hahn-Banach theorem, passing to the annihilator spaces gives a one-one correspondence between closed, $\alpha$-invariant subspaces of $\mathcal T^1(F^2(\mathbb C^d))$ and $\mathcal L(F^2(\mathbb C^d))$. 

For $A \in \mathcal T^1(F^2(\mathbb C^d))$ we define its Fourier (or Fourier-Weyl) transform by $\mathcal F_W(A)(\xi) = \tr(AW_\xi)$, where $\xi \in \mathbb C^d$. By the Riemann-Lebesgue lemma of Quantum Harmonic Analysis \cite[Prop.\ 3.4(6)]{werner84}, $\mathcal F_W(A) \in C_0(\mathbb C^d)$. Then, we set
\begin{align*}
    Z(A) = \{ \xi \in \mathbb R^{2d}: ~\mathcal F_W(A)(\xi) = 0\}.
\end{align*}
Again, using the fact that $\mathcal F_W(A)$ is continuous, $Z(A)$ is a closed subset of $\mathbb R^{2d}$. Further, for $X_1 \subset \mathcal T^1(F^2(\mathbb C^d))$ $\alpha$-invariant and closed, we set
\begin{align*}
    Z(X_1) = \bigcap_{A \in X_1} Z(A),
\end{align*}
which is still closed. Given $I \subset \mathbb C^{d}$ closed, we let
\begin{align*}
    X_{I,1} = \{ A \in \mathcal T^1(F^2(\mathbb C^d)): ~\mathcal F(A)(\xi) = 0 \text{ for every } \xi \in I\}.
\end{align*}
Then, it clearly is $Z(X_{I, 1}) = I$. We say that $I$ is a \emph{set of quantum spectral synthesis} if $Z(X_1) = I$ implies $X_1 = X_{I,1}$.

For $Y_1 \subset \mathcal L(F^2(\mathbb C^d))$ $\alpha$-invariant and weak$^\ast$ closed, we let
\begin{align*}
    \Sigma(Y_1) = \{ \xi \in \mathbb R^{2d}: ~W_\xi \in Y_1\}.
\end{align*}
Then, as in the function case, we obtain:
\begin{lemma}
Let $X_1 \subset \mathcal T^1(F^2(\mathbb C^d))$ be a closed, $\alpha$-invariant subspace. Then, it holds true that:
\begin{align*}
    Z(X_1) = \Sigma(X_1^\perp).
\end{align*}
\end{lemma}
\begin{proof}
It is
\begin{align*}
    \xi \in Z(X_1) \Leftrightarrow \tr(AW_\xi) = 0, A \in X_1 \Leftrightarrow W_\xi \in X_1^\perp \Leftrightarrow \xi \in \Sigma(X_1^\perp),
\end{align*}
which already finishes the proof.
\end{proof}
Given $I \subset \mathbb R^{2d}$ closed, we denote
\begin{align*}
    Y_{I, 1} = \overline{\operatorname{span}} \{ W_\xi: ~\xi \in I\},
\end{align*}
where the closure is of course taken in weak$^\ast$ topology. Then, $Y_{\Sigma(Y_1), 1} \subset Y_1$ for every closed, $\alpha$-invariant subspace $Y_1$ of $\mathcal L(F^2(\mathbb C^d))$ and $I$ is a set of quantum spectral synthesis if and only if $\Sigma(Y_1) = I$ implies $Y_1 = Y_{I, 1}$. Indeed, the problem of quantum spectral synthesis is equivalent to the classical problem of spectral synthesis:

\begin{theorem}\label{thm:specsynth_qspecsynth}
Let $I$ be a closed subset of $\mathbb R^{2d} \cong \mathbb C^d$. Then, $I$ is a set of quantum spectral synthesis if and only if it is a set of spectral synthesis.
\end{theorem}
Before attempting the proof, we need to introduce some more notation from Quantum Harmonic Analysis. Given $A, B \in \mathcal T^1(F^2(\mathbb C^d))$ and $f \in L^1(\mathbb C^d)$, we define the following two convolutions:
\begin{align*}
    f \ast A := A \ast f := \int_{\mathbb C^d} f(w) \alpha_z(A)~dz, \quad (A \ast B)(z) := \tr(A W_z U B U W_{-z}),~z \in \mathbb C^d.
\end{align*}
Here, $U$ is the parity operator $Uf(z) = f(-z)$ on $F^2(\mathbb C^d)$. We refer to \cite{werner84, Fulsche2020} for properties of these convolutions but mention that $A \ast f \in \mathcal T^1(F^2(\mathbb C^d))$ and $A \ast B \in L^1(\mathbb C^d)$. By the convolution theorem of QHA \cite[Prop.\ 3.4(1)]{werner84}, it is $Z(f \ast A) = Z(f) \cup Z(A)$ and $Z(A \ast B) = Z(A) \cup Z(B)$. Indeed, the convolutions are still well-defined if merely $B \in \mathcal L(F^2(\mathbb C^d))$ or $f \in L^\infty(\mathbb C^d)$.

An operator $A \in \mathcal T^1(F^2(\mathbb C^d))$ is called \emph{regular} if $\mathcal F_W(A)(\xi) \neq 0$ for every $\xi \in \mathbb C^d$. Such regular operators play an important role in Werner's Correspondence Theory. Right now, we mention only one instance of this theory, another one will be mentioned later. For the proof of the following result, see \cite[Thm.\ 4.1]{werner84} or \cite[Theorem 2.21]{Fulsche2020} (where the latter formulates the results only for the correspondence between bounded uniformly continuous functions on $\mathbb C^n$ and the algebra $\mathcal C_1$ we are going to introduce below, but the proof works analogously for this case). In the following, we will write $A \ast X_0 := \{ A \ast f: ~f \in X_0\}$ and $A \ast X_1 := \{ A \ast B: ~B \in X_1 \}$ for $A \in \mathcal T^1(F^2(\mathbb C^d))$, $X_0 \subset L^1(\mathbb C^d)$ and $X_1 \subset \mathcal T^1(F^2(\mathbb C^d))$. Similar notions will be used later without further explanation and should cause no confusion.
\begin{theorem}[{\cite[Theorem 4.1]{werner84}}]\label{corr:l1}
\begin{enumerate}[(1)]
    \item There is a one-one correspondence between closed, $\alpha$-invariant subspaces of $L^1(\mathbb C^d)$ and $\mathcal T^1(F^2(\mathbb C^d))$, which is obtained as follows: Given a regular operator $A \in \mathcal T^1(F^2(\mathbb C^d))$ and a closed, $\alpha$-invariant subspace $X_0 \subset L^1(\mathbb C^d)$, the corresponding space is $X_1 = \overline{A \ast X_0}$. Given $X_1 \subset \mathcal T^1(F^2(\mathbb C^d))$ closed and $\alpha$-invariant, the corresponding space is $X_0 = \overline{A \ast X_1}$. The corresponding spaces are independent of the choice of the regular operator $A$.
    \item Let $X_0, X_1$ be such corresponding spaces and $A\in \mathcal T^1(F^2(\mathbb C^d))$ regular.
    \begin{itemize}
        \item If $f \in L^1(\mathbb C^d)$, then $f \in X_0$ if and only if $A \ast f \in X_1$.
        \item If $B \in \mathcal T^1(F^2(\mathbb C^d))$, then $B \in X_1$ if and only if $A \ast B \in X_0$.
    \end{itemize}
\end{enumerate}
\end{theorem}
\begin{proof}[Proof of Theorem \ref{thm:specsynth_qspecsynth}]
The statement is a direct consequence of the above correspondence theorem: Given a regular trace class operator $A$, the maps $f \mapsto A \ast f$ and $B \mapsto A \ast B$ give a one-one correspondence between closed, $\alpha$-invariant subspaces of $L^1(\mathbb R^{2d})$ and $\mathcal T^1(F^2(\mathbb C^d))$ as described above. Since we have (by the convolution theorem) $Z(A \ast f) = Z(f)$ and $Z(A \ast B) = Z(A)$, it is for closed, $\alpha$-invariant subspaces $X_0 \subset L^1(\mathbb R^{2d})$ and $X_1 \subset \mathcal T^1(F^2(\mathbb C^d))$: $Z(\overline{X_0 \ast A_0}) = Z(X_0)$ and $Z(\overline{X_1 \ast A_0}) = Z(X_1)$. In particular, if there is only one closed, $\alpha$-invariant subspace $X_0$ with $Z(X_0) = I$, then there can be only one closed, $\alpha$-invariant subspace $X_1$ with $Z(X_1) = I$ and vice versa.
\end{proof}

We will also mention the following two facts about Quantum Harmonic Analysis which are not strictly speaking part of the theory of spectral synthesis, but fit in here quite nicely and will turn out very useful in our applications. For the first result, we write
\begin{align*}
    \mathcal C_1 := \{ A \in \mathcal L(F^2(\mathbb C^d)): ~\alpha_z(A) \to A \text{ in operator norm as } z \to 0\}.
\end{align*}
Clearly, $\mathcal C_1$ is a $C^\ast$-subalgebra of $\mathcal L(F^2(\mathbb C^d))$. Note that all the Weyl operators $W_z$ are contained in $\mathcal C_1$, whereas the parity operator $U$ is not contained in $\mathcal C_1$.
\begin{proposition}\label{prop:wstarclosure}
Let $Y_1 \subset \mathcal L(F^2(\mathbb C^d))$ be a weak$^\ast$ closed, $\alpha$-invariant subspace. Then, the weak$^\ast$ closure of $Y_1 \cap \mathcal C_1$ equals $Y_1$.
\end{proposition}
\begin{proof}
Clearly, the weak$^\ast$ closure of $Y_1 \cap \mathcal C_1$ is contained in $Y_1$. For proving that they are equal, let $B \in Y_1$. Below, $(g_t)_{t > 0}$ may be any positive approximate unit in $L^1(\mathbb R^{2d})$ with $g_t(-x) = g_t(x)$. For convenience, we let $g_t(x) = \frac{1}{(\pi t)^d}e^{-\frac{|x|^2}{t}}$. We will show the following facts:
\begin{itemize}
    \item $g_t \ast B \in Y_1 \cap \mathcal C_1$;
    \item $g_t \ast B \to B$ in weak$^\ast$ topology as $t \to 0$. 
\end{itemize}
Here, the convolution $g_t \ast B$ is defined by the same formula as above, even though $B$ is not necessarily trace class. Since $z \mapsto \alpha_z(B)$ is continuous in weak$^\ast$ topology, the integral also exists in this sense. Since $\overline{Y_1 \cap \mathcal C_1}$ is closed, these two facts then show that $\overline{Y_1 \cap \mathcal C_1} = Y_1$. 

For the first point, it is clear that $g_t \ast B \in \mathcal C_1$ \cite{werner84, Fulsche2020}, hence we only need to verify membership in $Y_1$. Let $A \in Y_1^\perp$. Then,
\begin{align*}
    \langle g_t \ast B, A\rangle &= \langle B, g_t \ast A\rangle = \tr(B \int_{\mathbb R^{2d}} g_t(x) \alpha_x(A)~dx)\\
    &= \int_{\mathbb R^{2d}} g_t(x) \tr(B\alpha_x(A))~dx = 0,
\end{align*}
since $\langle B, \alpha_x(A)\rangle = 0$ for every $x$. Hence, $g_t \ast B \in Y_1^{\perp \perp} = Y_1$.

For the second fact, note that for any $A \in \mathcal T^1(F^2(\mathbb C^d))$ it is $g_t \ast A \to A$ in $\mathcal T^1(F^2(\mathbb C^d))$. Further, using that $\langle g_t \ast B, A\rangle = \langle B, g_t \ast A\rangle$ (cf.\ \cite{werner84, Fulsche2020}, this uses $g_t(-x) = g_t(x)$) it is:
\begin{align*}
    |\langle g_t \ast B - B, A\rangle| = |\langle B, g_t \ast A - A\rangle| \leq \| B\|_{op} \| g_t \ast A - A\|_{\mathcal T^1} \to 0, ~t \to 0. 
\end{align*}
Since this holds for every $A \in \mathcal T^1(F^2(\mathbb C^d))$, we obtain $g_t \ast B \to B$ in weak$^\ast$ topology.
\end{proof}

Above we have already mentioned the Correspondence Theorem in its $L^1$-$\mathcal T^1$ form. Werner described in \cite{werner84} also its $\operatorname{BUC}(\mathbb C^d)$-$\mathcal C_1$ form (where $\operatorname{BUC}(\mathbb C^d)$ is the set of bounded uniformly continuous functions on $\mathbb C^d$), which is particularly useful in the theory of Toeplitz operators \cite{Fulsche2020}. We briefly want to end this part of the paper by mentioning that there is also an $L^\infty(\mathbb C^d)$-$\mathcal L(F^2(\mathbb C^d))$ version of the correspondence theorem, which of course was also already contained in \cite{werner84}. For independence of our presentation, we also sketch the short proof.
\begin{theorem}[{\cite[Cor.\ 4.4(2)]{werner84}}]\label{corr:linfty}
\begin{enumerate}[(1)]
    \item There is a one-one correspondence between closed (in weak$^\ast$ topology), $\alpha$-invariant subspaces of $L^\infty(\mathbb C^d)$ and $\mathcal L(F^2(\mathbb C^d))$, which is obtained as follows: Given a regular operator $A \in \mathcal T^1(F^2(\mathbb C^d))$ and a closed, $\alpha$-invariant subspace $Y_0 \subset L^\infty(\mathbb C^d)$, the corresponding space is $Y_1 = \overline{A \ast Y_0}$ (with closure in weak$^\ast$ topology). Given $Y_1 \subset \mathcal L(F^2(\mathbb C^d))$ closed and $\alpha$-invariant, the corresponding space is $Y_0 = \overline{A \ast Y_1}$. The corresponding spaces are independent of the choice of the regular operator $A$.
    \item Let $Y_0, Y_1$ be such corresponding spaces and $A\in \mathcal T^1(F^2(\mathbb C^d))$ regular.
    \begin{itemize}
        \item If $f \in L^\infty(\mathbb C^d)$, then $f \in Y_0$ if and only if $A \ast f \in Y_1$.
        \item If $B \in \mathcal L(F^2(\mathbb C^d))$, then $B \in Y_1$ if and only if $A \ast B \in Y_0$.
    \end{itemize}
    \item If $Y_0 \subset L^\infty(\mathbb C^d)$ and $Y_1 \subset \mathcal L(F^2(\mathbb C^d))$ are closed, $\alpha$-invariant subspaces, then they are corresponding spaces if and only if $Y_0^\perp$ and $Y_1^\perp$ are corresponding spaces in the sense of Theorem \ref{corr:l1}.
\end{enumerate}
\end{theorem}
\begin{proof}
The proofs for (1) and (2) work essentially as for the $L^1$-$\mathcal T^1$ or the $\operatorname{BUC}$-$\mathcal C_1$ correspondence theorems, with the sole difference that the approximation $g_t \ast B$ or $g_t \ast f$ (with $g_t$ e.g. as in the proof of Proposition \ref{prop:wstarclosure}) now converges in weak$^\ast$ topology to $B$ or $f$, respectively. As a next step, we fix a regular operator $A \in \mathcal T^1(F^2(\mathbb C^d))$. Now, for fixed $t > 0$, Wiener's approximation theorem for operators \cite[Proposition 3.5]{werner84} yields that $g_t$ can be approximated in $L^1(\mathbb C^d)$-topology by a finite sum $\sum_k c_k \alpha_{z_k}(A \ast A)$ such that $B \approx g_t \ast B \approx B \ast (\sum_k c_k \alpha_{z_k}(A \ast A))$, where the approximation error can be made arbitrarily small in weak$^\ast$ topology. The same procedure gives an approximation of $f$. This approximation procedure essentially yields the statement (2). Repeating the same argument with a different regular operator in place of $A$ shows that the correspondence in (1) is independent of the choice of $A$. The technical details to this are as in \cite{werner84} or \cite{Fulsche2020}. 

For proving (3), assume that $Y_0$ and $Y_1$ are corresponding spaces. Further, let $A \in \mathcal T^1(F^2(\mathbb C^d))$ be regular such that $UAU = A$. Then, the space corresponding to $Y_0^\perp$ is, since the functions $A \ast B$ with $B \in Y_1$ are dense in $Y_0$:
\begin{align*}
    \overline{A \ast Y_0^\perp} &= \overline{\{ A \ast f: ~\langle f, g\rangle = 0 \text{ for all } g \in Y_0\}}\\
    &= \overline{\{ A \ast f: ~\langle f, A \ast B\rangle = 0 \text{ for all } B \in Y_1\}}\\
    &= \overline{\{ A \ast f: ~\langle A \ast f, B\rangle = 0 \text{ for all } B \in Y_1\}}\\
    &\subset Y_1^\perp.
\end{align*}
Analogously, one shows that $\overline{A \ast Y_1^\perp} \subset Y_0^\perp$. As in \cite{werner84} or \cite{Fulsche2020}, this implies $\overline{A \ast Y_1^\perp} = Y_0^\perp$ and $\overline{A \ast Y_0^\perp} = Y_1^\perp$, i.e. the spaces are corresponding.

If we assume instead that $Y_0^\perp$ and $Y_1^\perp$ are corresponding, we can take the same approach, using that $Y_0 = Y_0^{\perp \perp}$ and $Y_1 = Y_1^{\perp \perp}$.
\end{proof}

\section{$G$-invariant Toeplitz algebras on the Fock space}

Let $G \subset \mathbb C^{d}$ be any subset. Set
\begin{align*}
    \mathcal A_G = \{ A \in \mathcal L(F^2(\mathbb C^d)): ~\alpha_x(A) = A \text{ for every } x \in G\}.
\end{align*}
Since $x \mapsto \alpha_x(A)$ is continuous (in weak$^\ast$ topology), we may assume that $G$ is closed (because $\mathcal A_G = \mathcal A_{\overline{G}}$). Further, for $x, y \in G$ we have
\begin{align*}
    \alpha_{x+y}(A) = \alpha_x(\alpha_y(A)) = \alpha_x(A) = A
\end{align*}
and
\begin{align*}
    A = \alpha_{x-x}(A) = \alpha_{-x}(\alpha_x(A)) = \alpha_{-x}(A).
\end{align*}
Therefore, we may always assume in the following that $G$ is a closed, additive subgroup of $\mathbb C^{d}$. Of course, considering algebras with such invariances is not a new idea, at least when $G$ is a lattice. To give just one reference, in \cite{Feichtinger_Kozek1998} such algebras have been investigated from the point of view of time-frequency analysis.

One easily verifies that $\mathcal A_G$ is an $\alpha$-invariant and weak$^\ast$ closed subspace of $\mathcal L(F^2(\mathbb C^d))$. We leave it to the interested reader to verify that the corresponding subspace in the sense of Theorem \ref{corr:linfty} is
\begin{align*}
    L^\infty(\mathbb C^d)_G := \{ f \in L^\infty(\mathbb C^d): ~\alpha_z(f) = f \text{ for every } z \in G\}.
\end{align*}
Recall that the Berezin transform of $B \in \mathcal L(F^2(\mathbb C^d))$ is given by $\widetilde{B}(z) = \langle Bk_z, k_z\rangle$. Since $\widetilde{B}$ is just the convolution of $B$ with the regular operator $1 \otimes 1$ (cf. \cite{Fulsche2020}), we obtain from Theorem \ref{corr:linfty} that $B \in \mathcal A_G$ if and only if $\widetilde{B} \in L^\infty(\mathbb C^d)_G$.

The first step towards understanding $\mathcal A_G$ is the following:
\begin{lemma}\label{lemm:whichweyl}
Let $G \subset \mathbb C^{d}$ be a closed subgroup and $z \in \mathbb C^{d}$. Then, $W_z \in \mathcal A_G$ if and only if $\sigma(z, w) \in 2\pi \mathbb Z$ for every $w \in G$.
\end{lemma}
\begin{proof}
For $w \in G$ we have:
\begin{align*}
    \alpha_w(W_z) = W_w W_z W_{-w} = e^{-\frac{i}{2}\sigma(w, z)}W_{w+z}W_{-w} = e^{-i\sigma(w,z)} W_z.
\end{align*}
Hence, $\alpha_w(W_z) = W_z$ if and only if $e^{-i\sigma(w, z)} = 1$, i.e.\ $\sigma(w,z) \in 2\pi \mathbb Z$, which is equivalent to $\sigma(z,w) \in 2\pi \mathbb Z$.
\end{proof}
Motivated by the previous lemma, we set
\begin{align}\label{def: G perp}
    G^\sigma := \{ z \in \mathbb C^{d}: \sigma(z, w) \in 2\pi \mathbb Z \text{ for every } w \in G\}.
\end{align}
Here are some basic properties of $G^\sigma$:
\begin{lemma}\label{properties:gperp}
Let $G, H$ be two closed additive subgroups of $\mathbb C^{d}$.
\begin{enumerate}[(1)]
    \item $G^\sigma$ is also a closed additive subgroup.
    \item $G^\sigma$ is a real linear subspace of $\mathbb C^{d}$ if and only if $G$ is a real linear subspace. In this case, $G^\sigma$ is precisely the symplectic complement of $G$.
    \item $G \subset H$ if and only if $H^\sigma \subset G^\sigma$.
    \item $(G^\sigma)^\sigma = G$.
    \item $(G + H)^\sigma = G^\sigma \cap H^\sigma$.
    \item When $S$ is a linear symplectomorphism of $(\mathbb R^{2d}, \sigma)$, then $(SG)^\sigma = S(G^\sigma)$.
\end{enumerate}
\begin{proof}
Only the fourth point needs proof, the other statements are then easily verified ((5) follows as for the symplectic complement, using (4))

It is easy to verify that $G \subset (G^\sigma)^\sigma$. Equality is indeed a consequence of the Pontryagin duality theorem. Upon identifying $\mathbb C^d \cong \mathbb R^{2d}$ with its character group $\widehat{\mathbb R^{2d}}$ via $\xi \mapsto e^{i\sigma(\xi, \cdot)}$, $G^\sigma$ is identified with the annihilator $\operatorname{Ann}(\widehat{\mathbb R^{2d}}, G)$. Doing this again, $(G^{\sigma})^{\sigma}$ is being identified with $\operatorname{Ann}(\widehat{\widehat{\mathbb R^{2d}}}, \operatorname{Ann}(\widehat{\mathbb R^{2d}}, G))$. This is now well-known to be equal to $G$, cf.\ \cite[Theorem 24.10]{Hewitt_Ross1}.  
\end{proof}
\end{lemma}

Note that $\mathcal A_G$ is clearly weak$^\ast$ closed and $\alpha$-invariant. Lemma \ref{lemm:whichweyl} can now be rephrased as:
\begin{align*}
    \Sigma(\mathcal A_G) = G^\sigma.
\end{align*}
The following fact is the only deep statement from classical harmonic analysis that we will use in this paper:
\begin{theorem}[{\cite[(40.24)]{Hewitt_Ross2}}]
Closed additive subgroups of $\mathbb C^{d}$ are sets of spectral synthesis.
\end{theorem}
Since spectral synthesis is equivalent to quantum spectral synthesis by Theorem \ref{thm:specsynth_qspecsynth}, we see that:
\begin{corollary}
    $\mathcal A_G = \overline{\operatorname{span}}\{ W_z: ~z \in G^\sigma\}$.
\end{corollary}
Given a subspace $\mathcal A$ of $\mathcal L(F^2(\mathbb C^d))$ we denote by $\mathcal A'$ the commutant:
\begin{align*}
    \mathcal A' := \{ B \in \mathcal L(F^2(\mathbb C^d)): [A, B] = 0 \text{ for all } A \in \mathcal A\}.
\end{align*}
Since $[A, \alpha_z(B)] = \alpha_z([\alpha_{-z}(A), B])$, $\mathcal A'$ is $\alpha$-invariant whenever $\mathcal A$ is so. Further, $\mathcal A'$ is weak$^\ast$ closed: Given a net $B_\gamma$ in $\mathcal A'$ converging to $B \in \mathcal L(F^2(\mathbb C^d))$ in weak$^\ast$ topology, it is for every $C \in \mathcal T^1(F^2(\mathbb C^d))$:
\begin{align*}
    0 = \langle [A, B_\gamma], C\rangle = \tr(CAB_\gamma) - \tr(ACB_\gamma) \to \tr(CAB) - \tr(ACB) = \langle [A, B], C\rangle.
\end{align*}
Hence, $\langle [A, B], C \rangle =0$ for every $C$ in trace class, hence $[A, B] = 0$ for every $A \in \mathcal A$. Therefore, we can try to understand $\mathcal A_G'$ in terms of its spectrum:
\begin{corollary}
$\Sigma(\mathcal A_G') = G$.    
\end{corollary}
\begin{proof}
We have $W_\xi \in \mathcal A_G'$ if and only if $[W_\xi, W_z] = 0$ for every $z \in G^\sigma$. It is
\begin{align*}
    [W_\xi, W_z] = W_\xi W_z - W_z W_\xi = e^{-i\sigma(\xi, z)}W_z W_\xi - W_z W_\xi.
\end{align*}
Hence, $[W_\xi, W_z] = 0$ if and only if $\sigma(\xi, z) \in 2\pi \mathbb Z$ for every $z \in G^\sigma$, i.e. $\xi \in (G^\sigma)^\sigma = G$.
\end{proof}
\begin{corollary}
$\mathcal A_G' = \mathcal A_{G^\sigma}$.
\end{corollary}
$\mathcal A_G$ is a $C^\ast$-algebra, which follows immediately from its definition (it is even a von Neumann algebra).
\begin{corollary}
$\mathcal A_G$ is commutative if and only if $G\supset G^\sigma$.
\end{corollary}
Note that closed subgroups $G$ with $G \supset G^\sigma$ are usually referred to as \emph{co-isotropic subgroups}.

As a last observation on the algebras $\mathcal A_G$, we want to note that indeed every $\alpha$-invariant von Neumann subalgebra of $\mathcal L(F^2(\mathbb C^d))$ is of this form.

\begin{proposition}
    Let $\mathcal A \subset \mathcal L(F^2(\mathbb C^d))$ be a von Neumann subalgebra which is $\alpha$-invariant. Then, $\mathcal A = \mathcal A_G$ for some closed subgroup $G$ of $\mathbb C^d$.
\end{proposition}
\begin{proof}
    Set $G = \{ z \in \mathbb C^d: ~W_z \in \mathcal A'\}$. It is immediate that $G$ is a closed subgroup of $\mathbb C^d$. Since $\mathcal A$ is $\alpha$-invariant, $\mathcal A'$ is a weak$^\ast$ closed subspace of $\mathcal L(F^2(\mathbb C^d))$ (even a von Neumann algebra itself) which is $\alpha$-invariant. Hence, $G = \Sigma(\mathcal A')$. Since $G$ is a closed subgroup, it is a set of quantum spectral synthesis, which yields $\mathcal A' = \mathcal A_{G^\sigma}$. Hence, $\mathcal A = \mathcal A'' = \mathcal A_{G^\sigma}' = \mathcal A_{G}$.
\end{proof}

Let us now pass to Toeplitz algebras. Recall that for $f \in L^\infty(\mathbb C^d)$, the Toeplitz operator $T_f \in \mathcal L(F^2(\mathbb C^d))$ is defined by $T_f(g) = P(fg)$. Here, $P$ is the orthogonal projection from $L^2(\mathbb C^d, \mu)$ to $F^2(\mathbb C^d)$. Toeplitz operators are directly related to Quantum Harmonic Analysis by the equality $\pi^d T_f = (1 \otimes 1) \ast f$ for every $f \in L^\infty(\mathbb C^d)$, cf.\ \cite[Proposition 2.12]{Fulsche2020}. By \cite[Theorem 3.1]{Fulsche2020}, we have
\begin{align*}
    \mathcal C_1 = C^\ast (\{ T_f: ~f \in L^\infty(\mathbb C^d)\}) = \overline{\{T_f: ~f \in \operatorname{BUC}(\mathbb C^d)\}}.
\end{align*}
Here, $C^\ast(S)$ denotes the $C^\ast$-algebra generated by the set $S \subset \mathcal L(F^2(\mathbb C^d))$. For a closed subgroup $G$ of $\mathbb C^{d}$, we have as an easy application of the $\operatorname{BUC}$-$\mathcal C_1$ version of the Correspondence Theorem (see \cite[Theorem 4.1]{werner84} or \cite[Proposition 3.3]{Fulsche2020} for its reformulation tailored for the application on Toeplitz operators):
\begin{proposition} For every closed additive subgroup $G$ of $\mathbb C^d$:
\begin{align*}
    \mathcal C_1 \cap \mathcal A_G &= \overline{\{ T_f: ~f \in \operatorname{BUC}(\mathbb C^d), ~\alpha_z(f) = f \text{ for every } z \in G\}}\\
    &= C^\ast ( \{ T_f: ~f \in L^\infty(\mathbb C^d), ~\alpha_z(f) = f \text{ for every } z \in G\}).
\end{align*}
\end{proposition}
\begin{remark}
If $G \subset \mathbb C^d$ is a lattice of full rank, then we can say even more. In that case, the C$^\ast$-algebra $\{ f \in \operatorname{BUC}(\mathbb C^d): \alpha_z(f) = f \text{ for every } z \in G\}$ coincides with the C$^\ast$-algebra generated by $\{ e^{i\sigma(z, \cdot)}: ~z \in G^\sigma\}$: Indeed, the latter is clearly a subalgebra of the former. Further, the latter separates points of the fundamental domain of the lattice (this follows immediately from the definition of $G^\sigma$), and the fundamental domain is essentially the Gelfand spectrum of the former, such that the Stone-Weierstrass theorem implies equality of both algebras. Another simple application of the $\operatorname{BUC}$-$\mathcal C_1$ version of the Correspondence Theorem shows that the first algebra corresponds to $\mathcal C_1 \cap \mathcal A_G$ and the latter algebra corresponds to the C$^\ast$-algebra generated by $\{ W_z: ~z \in G^\sigma\}$. Now this last algebra is exactly a noncommutative torus in the sense of noncommutative geometry. The structure of such C$^\ast$-algebras has of course been investigated in great details. As just one reference, we mention the work of Rieffel \cite{Rieffel1988}. There, he for example computed the coupling constants of the von Neumann algebras generated by noncommutative tori, which are exactly our algebras $\mathcal A_G$ for $G$ a lattice. In the end of Section 2 of that work, Rieffel already points out some connection between noncommutative tori and aspects of representation theory of the Heisenberg group, which could be seen as a connection to our approach through spectral synthesis.
\end{remark}

Clearly, we have that $\mathcal C_1 \cap \mathcal A_G$ is commutative if and only if $\mathcal A_G$ is commutative (the algebra is commutative if and only if the Weyl operators contained in it commute, and all the Weyl operators are contained in $\mathcal C_1$). Hence, this approach gives a wealth of commutative Toeplitz algebras:
\begin{corollary}
$C^\ast(\{ T_f: ~f \in L^\infty(\mathbb C^d), ~\alpha_z(f) = f \text{ for every } z \in G\})$ is commutative if and only if $G \supset G^\sigma$.
\end{corollary}
The case were $G$ contains a Lagrangian subspace of $\mathbb R^{2d}$ is well-known to be commutative \cite{Esmeral_Vasilevski2016}, but there are more examples which seemingly have not been noted in the literature yet.
\begin{example}\label{ex:1}
Let $m_1, m_2 > 0$ and consider $G = m_1 \mathbb Z + m_2 i\mathbb Z$ as a closed subgroup in $\mathbb C \cong \mathbb R^2$. Then, 
\begin{align*}
    G^\sigma = \frac{\pi}{m_2}\mathbb Z + \frac{\pi}{m_1}i \mathbb Z.
\end{align*}
In particular, we have $G \supset G^\sigma$ if and only if $\frac{\pi}{m_1 m_2} \in \mathbb N$, with $G = G^\sigma$ if and only if $m_1 m_2 = \pi$. 
\end{example}
\begin{example}\label{ex:symp}
When $S$ is a linear symplectomorphism of $(\mathbb R^2, \sigma)$ and $G$ is as in the previous example (i.e.\ with $\frac{\pi}{m_1 m_2} \in \mathbb N$), then $\mathcal A_{SG}$ is also commutative.
\end{example}
\begin{example}\label{ex:symp2}
From the previous discrete examples, the standard Lagrangian example and (5) of Lemma \ref{properties:gperp}, one can construct many more examples in higher dimensions. For example, letting $e_1, \dots, e_4$ denote the standard basis vectors of $\mathbb R^4$,
\begin{align*}
    G = \mathbb R e_1 \times \{ 0 \} \times \sqrt{\pi} \mathbb Z e_3 \times \sqrt{\pi}\mathbb Z e_4
\end{align*}
satisfies $G = G^\sigma$.
\end{example}

Let us come to the center of $\mathcal A_G \cap \mathcal C_1$ in $\mathcal C_1$. The following result significantly generalizes \cite[Theorem 5.9]{Hernandez}:
\begin{proposition}
$(\mathcal C_1 \cap \mathcal A_G)' \cap \mathcal C_1 = \mathcal A_G' \cap \mathcal C_1 = \mathcal A_{G^\sigma} \cap \mathcal C_1$.
\end{proposition}
\begin{proof}
Clearly, an operator $A \in \mathcal L(F^2(\mathbb C^d))$ commutes with every element from $\mathcal A_G \cap \mathcal C_1$ if and only if it commutes with every element from the weak$^\ast$ closure of $\mathcal A_G \cap \mathcal C_1$. By Proposition \ref{prop:wstarclosure}, this is just $\mathcal A_G$. Hence, $(\mathcal A_G \cap \mathcal C_1)' = \mathcal A_G' = A_{G^\sigma}$.
\end{proof}

We want to add a brief discussion on the commutative algebras described in the above examples. Esmeral and Vasilevski mention in \cite{Esmeral_Vasilevski2016} that there are only two model cases of commutative Toeplitz algebras over the Fock space. Even in one complex dimension, Example \ref{ex:1} gives a third example, which serves as a model for all the commutative Toeplitz algebras $\mathcal A_{S(G)} \cap \mathcal C_1$ described in Example \ref{ex:symp}. Indeed, the reason for the statement of Esmeral and Vasilevski might be the following: Vasilevski's classical strategy for classifying commutative Toeplitz $C^\ast$-algebras, as successfully applied in the case of the Bergman space on the unit disk \cite{Vasilevski2008}, consists very roughly speaking in introducing a quantization parameter $t$ (in the case of the Bergman space usually called $\lambda$) such that the Berezin quantization $f \mapsto \widetilde{f}^{(t)}$ is in some sense just the identity in the limit $t \to 0$, cf. \cite{Vasilevski2008, Zhu2012} for details. Based on this, Vasilevski could in \cite{Vasilevski2008} classify all Toeplitz $C^\ast$-algebras which are (in some well-defined sense) commutative for every value of the parameter $t > 0$. Upon implementing this for the Fock space, the standard approach for this would be replacing the Gaussian measure by $d\mu_t(z) = \frac{1}{(\pi t)^d} e^{-\frac{|z|^2}{t}}~dz$. In this case, the symplectic form $\sigma$ has to be replaced by $\sigma_t(z,w) = \frac{2}{t}\im(z \cdot \overline{w})$. Thus, for a closed additive subgroup $G$ of $\mathbb C^d$, the right object to study now is the parameter-dependent annihilator group
\begin{align*}
    G^{\sigma_t} := \{ z \in \mathbb C^d: ~\sigma_t(z,w) \in 2\pi \mathbb Z\} = tG^\sigma.
\end{align*}
Hence, on the Fock space $F_t^2(\mathbb C^d) = \mathcal O(\mathbb C^d) \cap L^2(\mathbb C^d, ~\mu_t)$, the algebra
\begin{align*}
    \mathcal A_G^t = \{ A \in \mathcal L(F_t^2(\mathbb C^d)): W_z^t A W_{-z}^t = A \text{ for every } z \in G\}
\end{align*}
(cf.\ e.g.\ \cite{Fulsche2020} for the definition of the Weyl operators depending on the parameter $t > 0$) has the commutator
\begin{align*}
    (\mathcal A_G^t)' = \mathcal A_{G^{\sigma_t}}^t = \mathcal A_{tG^\sigma}^t,
\end{align*}
i.e.\ $\mathcal A_G^t$ is commutative if and only $G \supset tG^\sigma$. If $G$ is a real subspace of $\mathbb C^d$, then this is satisfied for one $t > 0$ if and only if it is satisfied for all $t > 0$. On the other hand, if we now pick the two-dimensional example $G = m_1 \mathbb Z + m_2 i \mathbb Z$, then $tG^\sigma = \frac{\pi t}{m_2}\mathbb Z + \frac{\pi t}{m_1}i \mathbb Z$ and $G \supset tG^\sigma$ for only a discrete set of $t > 0$. Letting for example $m_1 = m_2 = \sqrt{\pi}$, then $\mathcal A_G^t$ is commutative if and only if $t = k$ for $k = 1, 2, 3, \dots$. Since these values of $t$ of course do not accumulate at $0$, these commutative algebras cannot be captured by Vasilevski's method. Of course, there are analogous subgroups $G$ that make $\mathcal A_G^t$ commutative for smaller values of $t$, e.g.\ with $G_s = \sqrt{\pi s} \mathbb Z + \sqrt{\pi s} i \mathbb Z$ (where $s > 0$), $\mathcal A_{G_t}^t$ is commutative. But again, $s = t$ is always the smallest value of $s$ making $\mathcal A_{G_s}^t$ commutative, so the same obstacle occurs.

It is not hard to see that a co-isotropic subgroup $G$ (i.e.\ $G \supset G^\sigma$) always contains a minimal co-isotropic subgroup, and these groups are exactly those which satisfy $G = G^\sigma$, see e.g.\ \cite[Theorem 1.6]{Hannabuss1979}. Therefore, any commutative algebra $\mathcal A_H$ with $H \supsetneq H^\sigma$ appears as a subalgebra of a maximal commutative algebra $\mathcal A_G$ with $G = G^\sigma$. In particular, the Gelfand theory of such $\mathcal A_H$ can easily be deduced from that of $\mathcal A_G$. Hence, if we can characterize all such groups $G$ with $G = G^\sigma$, this will give us access to characterizing all commutative $C^\ast$-algebras which can be described as $\mathcal A_G$. As already noted before, symplectic matrices play an important role in this theory (which has manifested itself, for example, in Example \ref{ex:symp}). We add some details to this.

Given a symplectic matrix $S \in \operatorname{Sp}(\mathbb R^{2d}, \sigma)$, i.e.\ a real-valued matrix $S \in M(2d, \mathbb R)$ satisfying $\sigma(Sz, Sw) = \sigma(z,w)$ for all $z, w \in \mathbb C^d \cong \mathbb R^{2d}$, we obtain:
\begin{align*}
    W_{Sz}W_{Sw} = e^{-\frac{i}{2}\sigma(Sz, Sw)}W_{Sw+Sz} = e^{-\frac{i}{2}\sigma(z, w)}W_{Sw+Sz},
\end{align*}
i.e.\ the operators $W_{Sz}$ satisfy the CCR relations. By the theorem of Stone and von Neumann, there exists a unitary operator $U_S \in \mathcal L(F^2(\mathbb C^d))$ such that $W_{Sz}U_S = U_S W_z$ for all $z \in \mathbb C^d$, and this $U_S$ is uniquely determined up to a constant of absolute value $1$. Clearly, we also have $W_{-z}U_S^\ast = U_S^\ast W_{-Sz}$. For $f \in L^\infty(\mathbb C^d)$ we set $f_{S^{-1}}(v) := f(S^{-1}v)$. First of all, note that $f_{S^{-1}}$ is $z$-invariant if and only if $f$ is $S^{-1}z$-invariant, i.e.\ $f \in L^\infty(\mathbb C^d)_G$ if and only if $f_{S^{-1}} \in L^\infty(\mathbb C^d)_{S^{-1}G}$. Now, for $f \in L^\infty(\mathbb C^d)$ we have:
\begin{align*}
    \pi^d U_S T_f U_S^\ast &= U_S [(1 \otimes 1) \ast f]U_S^\ast = \int_{\mathbb C^d} f(w) U_S W_w (1 \otimes 1) W_{-w}U_S^\ast ~dw\\
    &= \int_{\mathbb C^d} f(w) W_{Sw} (U_S 1 \otimes U_S 1) W_{-Sw}~dw\\
    &= \int_{\mathbb C^d} f(S^{-1}w) W_w (U_S 1 \otimes U_S 1) W_{-w}~dw\\
    &= f_{S^{-1}} \ast (U_S 1 \otimes U_S 1).
\end{align*}
Here, we have used in the substitution in the above equations the fact that every symplectic matrix has determinant $1$. For many choices of $S$ it will happen that $U_S 1 \neq 1$, in which case $U_S T_f U_S^\ast$ is no longer a Toeplitz operator. Nevertheless, this does not cause any trouble: First of all, note that
\begin{align*}
    \mathcal F_W(U_S 1 \otimes U_S 1)(z) &= \langle W_z U_S 1, U_S 1\rangle = \langle U_S W_{S^{-1}z} 1, U_S 1\rangle \\
    &= \langle W_{S^{-1}z}1, 1\rangle = \mathcal F_W(1 \otimes 1)(S^{-1}z) \neq 0.
\end{align*}
Hence, $U_S1 \otimes U_S 1$ is a regular operator. Further, recall that the Correspondence Theorem \ref{corr:linfty} is independent of the particular choice of the regular operator. Combining these facts yields:
\begin{proposition}
Let $G \subset \mathbb C^d$ be a closed subgroup and $S \in Sp(\mathbb R^{2d}, \sigma)$. Then, $U_S \mathcal A_G U_S^\ast = \mathcal A_{S^{-1}G}$.
\end{proposition}
We will dedicate the remaining part of this paper to study the commutative algebras $\mathcal A_G \cap \mathcal C_1$. While in general $U_S \not \in \mathcal C_1$, adjoining with $U_S$ leaves $\mathcal C_1$ invariant. Since adjoining with a unitary operator is an isomorphism of $C^\ast$-algebras, the previous proposition reduces the problem to the study of $\mathcal A_G$ for just one $G$ from each orbit of $Sp(\mathbb R^{2d}, \sigma)$. 

For this, the following result comes in handy.

\begin{theorem}\label{thm:structure_lagrangian_subgroups}
    Let $G \subset \mathbb C^d$ be a \emph{minimal co-isotropic subgroup}, i.e. $G = G^\sigma$. Then, there is some $S \in \operatorname{Sp}(\mathbb R^{2d}, \sigma)$ such that
    \begin{align*}
       S G = (\mathbb R \oplus \{ 0\})^k \oplus (\sqrt{\pi}\mathbb Z \oplus \sqrt{\pi} \mathbb Z)^{d-k}.
    \end{align*}
\end{theorem}
This classification of minimal co-isotropic (or \emph{Lagrangian}) subgroups is certainly well-known. Since we could not locate a good reference, we provide a proof in the appendix.

Since the $U_S$ are unitary, it is now clear that the operator algebras
\begin{align*}
    \mathcal C_1 \cap \mathcal A_G, \quad G = (\mathbb R \oplus \{ 0\})^k \oplus (\sqrt{\pi} \mathbb Z \oplus \sqrt{\pi} \mathbb Z)^{d-k}
\end{align*}
serve as model spaces for all the commutative Toeplitz C$^\ast$-algebras described above. In the next section, we will describe the Gelfand theories of the model spaces and hence of all the C$^\ast$-algebras.

\begin{remark}
    After publishing our results as a preprint on arXiv, Nikolai Vasilevski made us aware that he already characterized the commutative Toeplitz algebras $C^\ast(\{ T_f: ~\alpha_x(f) = f \text{ for } x\in G\})$, where $G \subset \mathbb C^d$ is a discrete subgroup, many years ago by rather different methods, but never published his results. He also pointed out to us the nice observation that, in the discrete case, the commutativity of the algebra can equivalently be rephrased in terms of the volume of the fundamental domain of the group. Let us briefly derive such a fact from our results for the case $d = 1$; results for higher dimensions can be derived analogously. So let $G = z_1 \mathbb Z  +  z_2 \mathbb Z$ be a discrete subgroup of rank 2 of $\mathbb C$. By an appropriate unitary transform, we can assume that $z_1 = \lambda > 0$, and by another appropriate transform $z \mapsto r\re(z) + i\im(z)/r$ for some $r > 0$, we may assume that $z_1 = 1$. Note that both transforms are symplectic and leave the area of the triangle spanned by $z_1$ and $z_2$ invariant. We may write $z_2 = a + ib$ with $b \neq 0$. Let $w \in G^\sigma$. Then, we have $\sigma(w, z_1), \sigma(w, z_2) \in 2\pi \mathbb Z$, which yields $\im(w), \im(w\overline{z_2}) \in \pi \mathbb Z$. Write $\im(w) = m \pi$. Then, $\im(w \overline{z_2}) = \pi m a - x b \in \pi \mathbb Z$, where $x = \re(w)$. Solving this for $x$ yields $w = \pi \frac{m a - n }{b} + i\pi m$ for some $m, n \in \mathbb Z$. Assume that each such $w$ is contained in $G$, i.e.\ for each such $m, n \in \mathbb Z$ we can find $k_1, k_2 \in \mathbb Z$ such that:
    \begin{align*}
        \pi \frac{ma - n}{b} + i\pi m = k_1 + k_2 (a+ib).
    \end{align*}
    Comparing the imaginary parts, we see that $k_2 = \frac{\pi m}{b} \in \mathbb Z$, which needs to be solvable for all $m \in \mathbb Z$. Hence, we see that $b = \pm \pi$. Therefore, for $G \supset G^\sigma$ it is necessary that $G = \mathbb Z + (a\pm i\pi)\mathbb Z$ for some $a \in \mathbb R$. It is not hard to verify that indeed each such $G$ satisfies $G \supset G^\sigma$. Now, the area of the triangle spanned by $z_1 = 1$ and $z_2 = a+ib$ is of course $\frac{1}{2}|b|$. Hence, the algebra $\mathcal A_G \cap \mathcal C_1$ is commutative if and only if the area of the triangle spanned by $z_1$ and $z_2$ equals $\frac{\pi}{2}$.
\end{remark}

\section{Gelfand theory}
We start by describing the Gelfand theories of the model cases in the one-dimensional situations in detail.
\subsection{Gelfand theory for $G = \mathbb R \oplus \{ 0\}$}\label{subsec:horizontal}
In the situation of $d = 1$ and $G = \mathbb R \oplus \{ 0\}$, most facts about the Gelfand theory have been worked out by Esmeral and Vasilevski in \cite{Esmeral_Vasilevski2016}. Let us briefly describe their findings. But before doing so, we want to emphasize that they described these results in arbitrary dimension $d$, whereas we will restrict to the case $d = 1$ for the moment.

Let us denote
\begin{align*}
    U_1&: L^2(\mathbb C, \mu) \to L^2(\mathbb R^2),\\
    U_1&(f)(x,y) = \frac{1}{\sqrt{\pi}} e^{-\frac{x^2 + y^2}{2}}f(x+iy),\\
    U_2&: L^2(\mathbb R^2) = L^2(\mathbb R) \otimes L^2(\mathbb R) \to L^2(\mathbb R) \otimes L^2(\mathbb R) = L^2(\mathbb R^2),\\
    U_2&= I \otimes \mathcal F,\\
    U_3&: L^2(\mathbb R^2) \to L^2(\mathbb R^2),\\
    U_3&(g)(x,y) = g\left( \frac{x+y}{\sqrt{2}}, \frac{x-y}{\sqrt{2}}\right).
\end{align*}
Here, $\mathcal F$ is the Fourier transform:
\begin{align*}
    \mathcal F(\varphi)(y) = \frac{1}{\sqrt{2\pi}} \int_{\mathbb R} e^{-i\eta y} \varphi(\eta)~d\eta.
\end{align*}
Then, all these operators are unitary, and hence $U = U_3U_2U_1: L^2(\mathbb C, \mu) \to L^2(\mathbb R^2)$ is also unitary. The range $U(F^2(\mathbb C))$ is given by
\begin{align*}
    U(F^2(\mathbb C)) = L^2(\mathbb R) \otimes L_0,
\end{align*}
where $L_0$ is the one-dimensional subspace of $L^2(\mathbb R)$ spanned by the Gaussian $\ell(y) = \frac{1}{{\pi}^{1/4}}e^{-\frac{y^2}{2}}$. The orthogonal projection from $L^2(\mathbb R^2)$ onto $L^2(\mathbb R) \otimes L_0$ is given by $I \otimes P_0$, where $P_0$ is the orthogonal projection from $L^2(\mathbb R)$ onto $L_0$. This projection satisfies
\begin{align*}
    U PU^\ast = I\otimes P_0,
\end{align*}
where $P$ is the orthogonal projection from $L^2(\mathbb C, \mu)$ onto $F^2(\mathbb C)$. Defining the map $B_0$ as
\begin{align*}
    B_0&: L^2(\mathbb R) \to L^2(\mathbb R^2),\\
    B_0&(\varphi) = \varphi \otimes \ell,
\end{align*}
$B_0$ is an isometry satisfying
\begin{align*}
    B_0^\ast B_0 &= I: L^2(\mathbb R) \to L^2(\mathbb R),\\
    B_0 B_0^\ast &= I \otimes P_0: L^2(\mathbb R^2) \to L^2(\mathbb R^2).
\end{align*}
Now, for Weyl operators $W_{iy}$, it is:
\begin{align*}
    UW_{iy}U^\ast = M_{E_y},
\end{align*}
where
\begin{align*}
    E_y(x) = e^{-i\sqrt{2}xy}.
\end{align*}
More generally, given a $G$-invariant symbol $a \in L^\infty(\mathbb C)$ it is
\begin{align*}
    UT_a U^\ast &= M_{\gamma_a},\\
    \gamma_a(x) &= \frac{1}{\sqrt{\pi}} \int_{\mathbb R} a\left( \frac{y}{\sqrt{2}}\right) e^{-(x-y)^2}~dy, \quad x \in \mathbb R.
\end{align*}
Besides providing these formulas, Esmeral and Vasilevski proved that the range of $a \mapsto \gamma_a$ is dense in $\operatorname{BUC}(\mathbb R)$. This immediately implies that the Gelfand transform of $\mathcal C_1 \cap \mathcal A_G$ maps onto $C(\mathcal M(\operatorname{BUC}))$, where we denote by $\mathcal M(\operatorname{BUC}))$ the maximal ideal space of $\operatorname{BUC}(\mathbb R)$, considered as a compactification of $\mathbb R$. For some reason, Esmeral and Vasilevski didn't provide the general formula for the Gelfand transform, a gap which we will now fill.
\begin{theorem}
    Let $G = \{ 0\} \oplus \mathbb R$. Then, the Gelfand transform of $\mathcal C_1 \cap \mathcal A_G$ is given by $\Gamma: \mathcal C_1 \cap \mathcal A_G \to C(\mathcal M(\operatorname{BUC}))$, $\Gamma(A) = \gamma_A$, where the function $\gamma_A \in \operatorname{BUC}(\mathbb R)$ is given by:
    \begin{align*}
        \gamma_A(x) &=  \frac{1}{2\pi} e^{\frac{x^2}{2}}\int_{\mathbb R} \int_{\mathbb R} e^{-\frac{\eta^2}{2}\left( \frac{1}{2} - i\right) + \sqrt{2}i\eta x} \langle Ak_{i\eta}, k_{(x+y)/\sqrt{2}}\rangle ~d\eta ~dy.
    \end{align*}
\end{theorem}
\begin{proof}
    We only need to add the formula for $\gamma_A$. We already know that
    \begin{align*}
        UAU^\ast(\ell \otimes \ell)(x,y) = \gamma_A(x) \ell \otimes \ell(x,y),
    \end{align*}
    hence
    \begin{align*}
        \int_{\mathbb R} UAU^\ast(\ell \otimes \ell)(x,y) ~dy = \gamma_A(x) \ell(x) \int_{\mathbb R} \ell(y)~dy = \gamma_A(x) \ell(x) \sqrt{2}\pi^{1/4}.
    \end{align*}
    Since $\ell(x) \neq 0$ for all $x \in \mathbb R$, we get:
    \begin{align*}
        \gamma_A(x) = \frac{1}{\sqrt{2}\pi^{1/4}\ell(x)} \int_{\mathbb R} UAU^\ast(\ell \otimes \ell)(x,y)~dy.
    \end{align*}
    Hence, we only have to compute this integral. Since $U^\ast(\ell \otimes \ell) = 1$, we have to compute $UA(1)(x,y)$. For this, we find:
    \begin{align*}
        U_1A(1)(x,y) &= \frac{1}{\sqrt{\pi}} e^{-\frac{x^2 + y^2}{2}}A(1)(x+iy),\\
        U_2 U_1 A(1)(x,y) &= \frac{1}{\sqrt{\pi}} \frac{1}{\sqrt{2\pi}} e^{-\frac{x^2}{2}}\int_{\mathbb R} e^{-\frac{\eta^2}{2} + iy\eta} A(1)(x+i\eta)~d\eta,\\
        U_3U_2U_1A(1)(x,y) &= \frac{1}{\sqrt{2}\pi} e^{-\frac{(x+y)^2}{4}} \int_{\mathbb R} e^{-\frac{\eta^2}{2} + i\frac{x-y}{\sqrt{2}}\eta}A(1)\left (\frac{x+y}{\sqrt{2}}+i\eta \right)~d\eta.
    \end{align*}
    We therefore obtain:
    \begin{align*}
        \gamma_A(x) = \frac{1}{2\pi}e^{\frac{x^2}{4}} \int_{\mathbb R} e^{-\frac{xy}{2} - \frac{y^2}{4}} \int_{\mathbb R} e^{-\frac{\eta^2}{4} + i \frac{x-y}{\sqrt{2}} \eta}A(1)\left( \frac{x+y}{\sqrt{2}} + i\eta\right) ~d\eta ~dy.
    \end{align*}
    Using now
    \begin{align*}
        A(1)\left(\frac{x+y}{\sqrt{2}} + i\eta \right) &= \langle A1, K_{(x+y)/\sqrt{2} + i\eta}\rangle\\
        &= \langle A1, k_{(x+y)/\sqrt{2} + i\eta}\rangle e^{\frac{(x+y)^2}{4}+ i \frac{\eta^2}{2}}\\
        &= \langle A1, W_{i\eta} k_{(x+y)/\sqrt{2}}\rangle e^{\frac{(x+y)^2}{4}+ i \frac{\eta^2}{2} + i\sigma(i\eta, \frac{x+y}{\sqrt{2}})}\\
        &= \langle AW_{i\eta} 1,  k_{(x+y)/\sqrt{2}}\rangle e^{\frac{(x+y)^2}{4}+ i \frac{\eta^2}{2} + i\eta \frac{x+y}{\sqrt{2}}}\\
        &= \langle Ak_{i\eta}, k_{(x+y)/\sqrt{2}}\rangle e^{\frac{(x+y)^2}{4}+ i \frac{\eta^2}{2} + i\eta \frac{x+y}{\sqrt{2}}},
    \end{align*}
     one obtains the formula for $\gamma_A(x)$.
\end{proof}

\subsection{Gelfand theory for $G = \sqrt{\pi} \mathbb Z \oplus \sqrt{\pi} \mathbb Z$}\label{subsec:lattice}
Write $G$ as $G = \{ w_k: ~k = (k_1, k_2) \in \mathbb Z^2\}$, where $w_k = \sqrt{\pi}(k_1 + ik_2)$. Then, $G$ is the so-called \emph{von Neumann lattice}. By \cite[Lemma 5.7]{Zhu2012}, this lattice is a set of uniqueness for $F^2(\mathbb C)$, i.e.\ if $f \in F^2(\mathbb C)$ vanishes on $G$, then $f = 0$. This of course implies:
\begin{align}\label{eq:kernelsdense}
    F^2(\mathbb C) = \overline{\operatorname{span}} \{ K_{w_k}: k \in \mathbb Z^2\}.
\end{align}
Write $R_k = R_0 + w_k$ with $R_0 = [0, \sqrt{\pi}] \times [0, \sqrt{\pi}]$. Consider
\begin{align*}
    U_1&: L^2(\mathbb C, \mu) \to \ell^2(\mathbb Z^2, L^2(R_0, \mu)),\\
    U_1&f = (W_{w_{-k}}(f|_{R_k}))_{k \in \mathbb Z^2}.
\end{align*}
It is easily verified that $U_1$ is unitary with inverse
\begin{align*}
    U_1^\ast((f_k)_{k \in \mathbb Z^2})(z) = \sum_{k \in \mathbb Z^2} \chi_{R_k}(z) e^{z\overline{w_k} - \frac{|w_k|^2}{2}} f_k(z - w_k), \quad z \in \mathbb C.
\end{align*}
We note that
\begin{align*}
    \ell^2(\mathbb Z^2, L^2(R_0, \mu)) \cong L^2(R_0, \mu) \otimes \ell^2(\mathbb Z^2).
\end{align*}
Therefore, we let
\begin{align}
    U_2 = I \otimes F: L^2(R_0, \mu) \otimes \ell^2(\mathbb Z^2) \longrightarrow L^2(R_0, \mu) \otimes L^2(\mathbb T^2, \nu),
\end{align}
where $F: \ell^2(\mathbb Z) \to L^2(\mathbb T^2, \nu)$ is the Fourier transform,
\begin{align*}
    F((c_k)_{k \in \mathbb Z^2})(\lambda) = \sum_{k \in \mathbb Z^2} c_k \lambda^{-k}, \quad \lambda \in \mathbb T^2,
\end{align*}
and $\nu$ is the Haar measure on $\mathbb T^2$ normalized such that $\nu(\mathbb T^2) = 1$. Again, $U_2$ is unitary with
\begin{align*}
    U_2^\ast(h) = \left( \int_{\mathbb T^2} h(\cdot, \lambda)\lambda^k d\nu(\lambda) \right)_{k \in \mathbb Z^2}, \quad h \in L^2(R_0, \mu) \otimes L^2(\mathbb T^2, \nu).
\end{align*}
Now, $V = U_2 U_1$ is unitary from $L^2(\mathbb C, \mu)$ to $L^2(R_0, \mu) \otimes L^2(\mathbb T^2, \nu) \cong L^2(R_0 \times \mathbb T^2, \mu \otimes \nu)$ and acts by
\begin{align*}
    V(f)(z, \lambda) = \sum_{k \in \mathbb Z^2} e^{-z \overline{w_k} - \frac{|w_k|^2}{2}} f(z+w_k)\lambda^{-k}, \quad (z, \lambda) \in R_0 \times \mathbb T^2.
\end{align*}
The adjoint of $V$ acts by
\begin{align*}
    V^\ast(g)(z) = \sum_{k \in \mathbb Z^2} \chi_{R_k}(z) e^{z \overline{w_k} - \frac{|w_k|^2}{2}} \int_{\mathbb T^2} g(z-w_k, \lambda)\lambda^{-k} d\nu(\lambda).
\end{align*}
\begin{lemma}\label{lemm:41}
    For $j \in \mathbb Z^2$ it is $V W_{w_j} V^\ast = M_{\varphi_j}$, where
    \begin{align*}
        \varphi_j(z, \lambda) = \lambda^{-j}, \quad (z, \lambda) \in R_0 \times \mathbb T^2
    \end{align*}
    and the Weyl operator is considered as an operator on $L^2(\mathbb C, \mu)$.
\end{lemma}
\begin{proof}
    Follows from direct verification, using the above formula for $V^\ast$.
\end{proof}
\begin{lemma}
    $V$ maps $F^2(\mathbb C)$ onto the subspace generated by functions of the form
    \begin{align*}
        p(\lambda, \lambda^{-1})h(z, \lambda), \quad (z, \lambda) \in R_0 \times \mathbb T^2,
    \end{align*}
    where $p(\cdot, \cdot)$ is a polynomial and $h = V(1) \in L^2(R_0 \times \mathbb T^2, \mu \otimes \nu)$ is the function given by
    \begin{align*}
        h(z, \lambda) = \sum_{k \in \mathbb Z^2} e^{-z\overline{w_k} - \frac{|w_k|^2}{2}} \lambda^{-k}.
    \end{align*}
\end{lemma}
\begin{proof}
    The formula for $h$ follows by direct verification. That the span of the functions $p(\lambda, \lambda^{-1})h(z, \lambda)$ is dense in $VF^2(\mathbb C)$ follows from the previous lemma and \eqref{eq:kernelsdense}.
\end{proof}
Note that
\begin{align*}
    h(z, \lambda) = \vartheta(z_1, \frac{i}{2}) \vartheta(z_2, \frac{i}{2}),
\end{align*}
where
\begin{align*}
    \lambda_1 = e^{i\theta_1}, \quad \lambda_2 = e^{i\theta_2}, \quad z_1 = -\frac{i\theta_1 + \sqrt{\pi}z}{2\pi i}, \quad z_2 = \frac{i\sqrt{\pi}z - i\theta_2}{2\pi i}
\end{align*}
and
\begin{align*}
    \vartheta(w, \tau) = \sum_{m=-\infty}^\infty e^{\pi i m^2 \tau + 2\pi i m w}, \quad w \in \mathbb C, \im(\tau) > 0
\end{align*}
is one of Jacobi's theta functions. We want to emphasize that the occurrence of a theta function is not at all surprising, see e.g.~\cite{Mumford_Nori_Norman1991} for a detailed discussion concerning the connection of the von Neumann lattice and theta functions.

Since $z_1, z_2$ depend continuously on $z$ and $\vartheta(w, \tau)$ is well-known to be continuous, $h(z, \lambda)$ continuously depends on $z$ (for fixed $\lambda)$. Further, the set of zeros of $\vartheta(w, \tau)$ is known to be discrete:
\begin{align*}
    \vartheta(w, \frac{i}{2}) = 0 \Leftrightarrow w= m + \frac{ni}{2} + \frac{1}{2} + \frac{i}{4}, \quad \text{ some } n, m \in \mathbb Z.
\end{align*}
This formula shows that $h(z,\lambda) = 0$ if and only if one of the following two equations is satisfied:
\begin{align*}
    \begin{cases} z_1 &= m + \frac{ni}{2} + \frac{1}{2} + \frac{i}{4}\\
    z_2 &= m + \frac{ni}{2} + \frac{1}{2} + \frac{i}{4}
    \end{cases}
    \Leftrightarrow
    \begin{cases}
        z &=  -\frac{\theta_1}{\sqrt{\pi}}i - 2\sqrt{\pi}m i + \sqrt{\pi}n + \sqrt{\pi}\frac{1-2i}{2}, \quad m, n \in \mathbb Z,\\
        z &= \frac{\theta_2}{\sqrt{\pi}} + 2\sqrt{\pi}m + \sqrt{\pi}ni + \sqrt{\pi}\frac{2 + i}{2}, \quad m, n \in \mathbb Z.
    \end{cases}
\end{align*}
Each of those equations only has a finite number of solutions in $R_0$. In particular, 
\begin{align*}
    H(\lambda) := \int_{R_0}|h(z, \lambda)|^2~d\mu(z) \in (0, \infty)
\end{align*}
for every $\lambda \in \mathbb T^2$. Denote by $\eta$ the measure $d\eta(\lambda) = H(\lambda)d\nu(\lambda)$ on $\mathbb T^2$. It is a probability measure, because $\eta(\mathbb T^2) = \| 1\|_{F^2} = 1$. Set
\begin{align*}
    S&: L^2(\mathbb T^2, \eta)\to L^2(R_0 \times \mathbb T^2, \mu \otimes \nu)\\
    S&(f)(\lambda, z)= f(\lambda)h(z, \lambda), \quad (z,\lambda) \in R_0 \times \mathbb T^2.
\end{align*}
$S$ is easily seen to be isometric. For $f \in L^2(\mathbb T^2, \eta)$ and $g \in L^2(R_0 \times \mathbb T^2, \mu \otimes \nu)$ we have
\begin{align*}
    \langle Sf, g\rangle &= \int_{R_0 \times \mathbb T^2} f(\lambda) h(z, \lambda) \overline{g(z, \lambda)}~d(\mu \otimes \nu)\\
    &= \int_{\mathbb T^2} f(\lambda) \overline{\left( \frac{1}{H(\lambda)}\int_{R_0} g(z, \lambda) \overline{h(z,\lambda)} d\mu(z) \right)} d\eta(\lambda),
\end{align*}
showing that
\begin{align*}
    S^\ast(g)(\lambda) = \frac{1}{H(\lambda)} \int_{R_0} g(z, \lambda) \overline{h(z, \lambda)}~d\mu(z).
\end{align*}
From here, one easily verifies that
\begin{align*}
    S^\ast S = I_{L^2(\mathbb T^2, \eta)},\quad SS^\ast = P_{V(F^2(\mathbb C))},
\end{align*}
where $P_{V(F^2(\mathbb C))} = VP_{F^2(\mathbb C)}V^\ast$ is the orthogonal projection onto the space $V(F^2(\mathbb C))$. 

\begin{proposition}
    Let $a \in L^\infty(\mathbb C)_G$. Then, the Toeplitz operator $T_a$ is unitarily equivalent to the multiplication operator $VT_a V^\ast = M_{\gamma_a}$ on $L^2(R_0 \times \mathbb T^2, \mu \otimes \nu)$, where $\gamma_a$ depends only on $\lambda$ and is given by
    \begin{align*}
        \gamma_a(\lambda)= \frac{1}{H(\lambda)} \int_{R_0} a(z) |h(z,\lambda)|^2 ~d\mu(z), \quad \lambda \in \mathbb T^2.
    \end{align*}
\end{proposition}
\begin{proof}
    It is
    \begin{align*}
        VT_a V^\ast = VP_{F^2(\mathbb C)} V^\ast VM_a V^\ast = SS^\ast M_{a_0},
    \end{align*}
    where $a_0(z,\lambda) = a(z)$ and $VM_a V^\ast = M_{a_0}$ is readily verified. Therefore, for every polynomial $\lambda \mapsto p(\lambda, \lambda^{-1})$ and $\lambda \in\mathbb T^2$, we have:
    \begin{align*}
        VT_a^\ast V^\ast(p \cdot h)(z, \lambda) &= SS^\ast (a_0 \cdot p \cdot h)(z, \lambda)\\
        &= h(z, \lambda) \cdot \frac{1}{H(\lambda)} \int_{R_0} a(z)p(\lambda, \lambda^{-1})|h(z, \lambda)|^2 ~d\mu(z)\\
        &= \gamma_a(\lambda)p(\lambda, \lambda^{-1})h(z, \lambda).
    \end{align*}
    Since functions of the form $p(\lambda, \lambda^{-1}) \cdot h(z, \lambda)$ are dense in $V(F^2(\mathbb C))$, the result follows.
\end{proof}
\begin{remark}
    Since the Weyl operators $W_{w_j}$ can be written as Toeplitz operators, $W_{w_j} = T_{g_j}$, where
    \begin{align*}
        g_j(z) = e^{2i\sigma(z, w_j) + \frac{|w_j|^2}{2}},
    \end{align*}
    the above result together with Lemma \ref{lemm:41}, gives the curious identity
    \begin{align*}
        H(\lambda) \lambda^{-j} = \int_{R_0}  e^{2i\sigma(z, w_j) + \frac{|w_j|^2}{2}} |h(z, \lambda)|^2~d\mu(z).
    \end{align*}
\end{remark}
\begin{theorem}
Let $G = \sqrt{\pi} \mathbb Z \oplus \sqrt{\pi}\mathbb Z$. The Gelfand spectrum of $\mathcal C_1 \cap \mathcal A_G$ is given by $\mathcal M(\mathcal C_1 \cap \mathcal A_G) \cong \mathbb T^2$ with Gelfand transform $A \mapsto \gamma_A$, where $\gamma_A$ is the continuous function which is given almost everywhere by
\begin{align*}
    \gamma_A(\lambda) = \frac{1}{H(\lambda)}\int_{R_0} \overline{h(z, \lambda)}e^{\frac{|z|^2}{2}}\sum_{k \in \mathbb Z^2} \langle Ak_{-w_k}, k_z\rangle \lambda^{-k}~d\mu(z).
\end{align*}
\end{theorem}

In an abuse of notation, we will identify $\gamma_{T_a}$ and $\gamma_a$ when $a$ is $G$-invariant.
\begin{proof}
    We already know that $T_a \mapsto \gamma_a$ is a multiplicative linear map from the set of all sums of products of Toeplitz operators with $G$-invariant symbols to $C(\mathbb T^2)$. Such operators are dense in $\mathcal C_1 \cap \mathcal A_G$, hence the map extends by continuity to a map
    \begin{align*}
        \mathcal C_1 \cap \mathcal A_G \ni A \mapsto \gamma_A \in C(\mathbb T^2).
    \end{align*}
    This map is unital and also respects the adjoint, i.e.\ $\gamma_{A^\ast} = \overline{\gamma_A}$, i.e. $A \mapsto \gamma_A$ is the Gelfand transform of the unital C$^\ast$-algebra $\mathcal C_1 \cap \mathcal A_G$. Since the polynomials are contained in the range (being the image of the Weyl operators), the Stone-Weierstrass theorem shows that the range is all of $C(\mathbb T^2)$. We are left with computing the general formula for $\gamma_A$.

    Given $A \in \mathcal C_1 \cap \mathcal A_G$, $VAV^\ast = M_{\gamma_A}$. Since $\gamma_A$ does not depend on $z \in R_0$, we can compute $\gamma_A$ as
    \begin{align*}
        \gamma_A(\lambda) &= \frac{\langle (VAV^\ast h)(\cdot, \lambda), h(\cdot, \lambda)\rangle_{L^2(R_0, \mu)}}{\| h(\cdot, \lambda)\|_{L^2(R_0, \mu)}^2} \\
        &= \frac{1}{H(\lambda)} \int_{R_0} \overline{h(z, \lambda)} VAV^\ast(h)(z,\lambda)~d\mu(z).
    \end{align*}
    Now, we have, using the $G$-invariance of $A$:
    \begin{align*}
        VAV^\ast(h)(z, \lambda) &= \sum_{k \in \mathbb Z^2} W_{-w_k} A V^\ast(h)(z)\lambda^{-k}\\
        &= \sum_{k \in \mathbb Z^2} AW_{-w_k} (1)(z) \lambda^{-k}\\
        &= \sum_{k \in \mathbb Z^2} A(k_{-w_k})(z) \lambda^{-k} \\
        &= e^{\frac{|z|^2}{2}}\sum_{k \in \mathbb Z^2} \langle Ak_{-w_k}, k_z\rangle \lambda^{-k}. \qedhere
    \end{align*}
\end{proof}
\begin{remark}
    Since the Gelfand transform of polynomials of $W_{w_{(1,0)}}$ and $W_{w_{(0,1)}}$ generates all polynomials on $\mathbb T^2$, we obtain a result in the spirit of the work \cite{Rozenblum_Vasilevski2022}: The C$^\ast$-algebra $\mathcal C_1 \cap \mathcal A_G$ is obtained as
    \begin{align*}
        \mathcal C_1 \cap \mathcal A_G = \{ f(W_{w_{(1,0)}}, W_{w_{(0,1)}}): ~f \in C(\mathbb T^2)\},
    \end{align*}
    i.e.\ the commutative C$^\ast$-algebra is obtained through the joint continuous functional calculus of the operators $W_{w_{(1,0)}}, W_{w_{(0,1)}}$.
\end{remark}

\subsection{The general model case}
We will now describe the Gelfand theory for the case where $G = (\mathbb R \oplus \{ 0\})^{d-k} \oplus (\sqrt{\pi}\mathbb Z \oplus \sqrt{\pi}\mathbb Z)^k$. Since each commutative algebra $\mathcal C_1 \cap \mathcal A_G$ is equivalent to one of these particular cases, we therefore describe the Gelfand theory for all commutative C$^\ast$-algebras obtained in this way.

In principle, the Gelfand theory reduces to the previous theories by applying the Gelfand transforms to each complex coordinate separately. We will leave the details to the interested reader and only give the results. In this section, we will write $\mu_1$ for the Gaussian measure on the one-dimensional space $\mathbb C$ to distinguish it from the Gaussian measure $\mu$ on $\mathbb C^d$.

Letting $U: L^2(\mathbb C, \mu) \to L^2(\mathbb R^{2})$ as in Subsection \ref{subsec:horizontal} and $V$ as in Subsection \ref{subsec:lattice}. Then,
\begin{align*}
    U' := \underbrace{U \otimes \dots \otimes U}_{k \text{ times}} \otimes \underbrace{V \otimes \dots \otimes V}_{d-k \text{ times}}
\end{align*}
gives a unitary map:
\begin{align*}
    U': L^2(\mathbb C^d, \mu) \to L^2(\mathbb R^{2k}) \otimes L^2((R_0 \times \mathbb T^2)^{d-k}, (\mu \otimes \nu)^{d-k}).
\end{align*}
Based on the results for $U$ and $V$, it is not hard to explicitly write the range of $U'(F^2(\mathbb C^d))$ out, but we will omit this here. For 
\begin{align*}
    w = (0+ ix_1, \dots, 0 + ix_k, \sqrt{\pi}k_1^1 + i\sqrt{\pi}k_2^1, \dots, \sqrt{\pi}k_1^{d-k} + i\sqrt{\pi}k_2^{d-k}) \in G
\end{align*}
we obtain
\begin{align*}
    U'W_w (U')^\ast = M_{\psi_w},
\end{align*}
where $\psi_w$ is the function on $\mathbb R^k \times (R_0 \times \mathbb T^2)^{d-k}$ defined by
\begin{align*}
    \psi_w&(y_1, \dots, y_k, (z_1, \lambda_1), \dots, (z_{d-k}, \lambda_{d-k})) \\
    &= E_{x_1}(y_1) \cdot \dots \cdot E_{x_k}(y_k) \varphi_{(k_1^1, k_2^1)}(z_1, \lambda_1) \cdot \dots \cdot \varphi_{(k_1^{d-k}, k_2^{d-k})}(z_{d-k}, \lambda_{d-k})
\end{align*}
with the functions $E_{x_m}$ and $\varphi_{(k_1^m, k_2^m)}$ taken from Subsections \ref{subsec:horizontal} and \ref{subsec:lattice}, respectively. More generally, for $T_a$ with $a$ $G$-invariant, we can restrict $a$ to a function $a'$ on $\mathbb R^k \times R_0^{d-k}$ to receive
\begin{align*}
    U' T_a (U')^\ast = M_{\gamma_a},
\end{align*}
where $\gamma_a$ is given by:
\begin{align*}
    \pi^{k/2}& H(\lambda_1) \dots H(\lambda_{d-k}) \gamma_a(y_1, \dots, y_k, \lambda_1, \dots, \lambda_{d-k})\\
    &= \int_{R_0^{d-k}}\int_{\mathbb R^k} a'(\frac{y_1}{\sqrt{2}}, \dots, \frac{y_k}{\sqrt{2}}, z_1, \dots, z_{d-k}) \prod_{j=1}^k e^{-(x_j - y_j)^2}\\
    &\quad \quad \times \prod_{m=1}^{d-k}|h(z_m, \lambda_m)|^2 ~dy_1\dots~dy_k ~d\mu_1(z_1) \dots ~d\mu_1(z_{d-k}).
\end{align*}
We give now the most general result:
\begin{theorem}
    The Gelfand spectrum of $\mathcal C_1 \cap \mathcal A_G$ is given by $\mathcal M(\operatorname{BUC}(\mathbb R^k \times \mathbb T^{2(d-k)}))$. The Gelfand transform is given by
    \begin{align*}
        \Gamma(A) = \gamma_A,
    \end{align*}
    where
    \begin{align*}
        \pi^d &e^{-\frac{x_1^2 + \dots + x_k^2}{2}} H(\lambda_1)\dots H(\lambda_{d-k})\gamma_A(x_1, \dots, x_k, \lambda_1, \dots, \lambda_{d-k}) \\
        &= \int_{\mathbb R^k} \int_{\mathbb R^k} \int_{R_0^{d-k}} \overline{h(z_1, \lambda_1)}\cdot \hdots \cdot \overline{h(z_{d-k}, \lambda_{d-k})} \\
        &\quad \quad \cdot \sum_{(k_j^1, k_j^2) \in \mathbb Z^2, j = 1, \dots, d-k} e^{-\frac{\eta_1^2+\dots + \eta_k^2}{2}(\frac{1}{2} - i) + \sqrt{2}i(\eta_1 x_1 + \dots + \eta_k x_k) + \frac{|z_1|^2 + \dots + |z_{d-k}|^2}{2}}\\
        &\quad \quad \cdot \langle A k_{(i\eta_1, \dots, i\eta_k,w_{(k_1^1, k_2^1)}, \dots, w_{(k_1^{d-k}, k_2^{d-k})})}, k_{(\frac{x_1 + y_1}{\sqrt{2}}, \dots, \frac{x_k + y_k}{\sqrt{2}}, z_1, \dots, z_{d-k})}\rangle \\
        &\quad \quad \quad d\mu_1(z_1)~\dots~d\mu_1(z_{d-k})~d\eta_1 \dots ~d\eta_k ~dy_1 \dots ~dy_k.
    \end{align*}
\end{theorem}
\begin{remark}
    We want to emphasize that the Gelfand theory of $\mathcal A_G$ with $G = G^\sigma$ can be treated analogously. The formulas one receives for the Gelfand transforms are identical, with the difference that $\Gamma$ now maps from $\mathcal A_G$ to $L^\infty(\mathbb R^k \times (\mathbb T^2)^{d-k})$. This follows easily from the previous computations and the fact that the span of the Weyl operators $W_w$, $w \in G$, is dense in $\mathcal A_G$ in weak$^\ast$ topology.
\end{remark}

\appendix

\section{The structure of Lagrangian subgroups}
This appendix is dedicated to the proof of Theorem \ref{thm:structure_lagrangian_subgroups}. We will prove three lemmas, which together will yield a proof of the theorem.

In the following, we will use the already introduced notation in the obvious way on general symplectic vector spaces (of finite dimension): If $(V, \sigma_0)$ is a finite-dimensional real symplectic vector space and $G \subset V$ is a subgroup, we will still write
\begin{align*}
    G^{\sigma_0} = \{ z \in V: ~\sigma_0(z, w) \in 2\pi \mathbb Z \text{ for all } w \in G\}.
\end{align*}
In this more general setting, the symplectic complement has the same properties as in the concrete setting of $(\mathbb R^{2d}, \sigma)$.

 It is well-known that every closed subgroup $G \subset \mathbb R^{2d}$ can be written as a sum of a vector part and a discrete part: $G = G_{vec} \oplus G_{disc}$, where $G_{vec}$ and $G_{disc}$ are both subgroups of $\mathbb R^{2d}$, $G_{vec}$ being a subspace and $G_{disc}$ a purely discrete subgroup.
 \begin{lemma}
        Let $G \subset \mathbb C^d$ be a Lagrangian subgroup. Then, $G_{vec}$ and $G_{disc}$ are orthogonal with respect to the standard complex inner product on $\mathbb C^d$. There are orthogonal subspaces $U, V$ of $\mathbb C^d$ such that $(U, \sigma|_{U\times U})$ and $(V, \sigma|_{V \times V})$ are symplectic spaces, $(\mathbb C^d, \sigma) = (U, \sigma|_{U \times U}) \oplus (V, \sigma|_{V \times V})$, and $G_{vec} = U \cap G$, $G_{disc} = V \cap G$. Further, $G_{vec}$ is a Lagrangian subgroup of $U$ and $G_{disc}$ is a Lagrangian subgroup of $V$.
    \end{lemma}
    \begin{proof}
        For $x \in G_{vec}$ and $\lambda \in \mathbb R$ it is:
    \begin{align*}
        \sigma(ix, \lambda x) = \lambda \im(i|x|^2) = \lambda |x|^2,
    \end{align*}
    which can only be in $\pi \mathbb Z$ for all values of $\lambda$ if $x = 0$. Hence, $ix \not \in G = G^\sigma$ when $x \in G_{vec}$. Consider the complex subspace $U = G_{vec} \oplus i G_{vec}$ of $\mathbb C^d$. It is not hard to see that $(U, \sigma|_{U \times U})$ is again a symplectic vector space, and $G_{vec}$ is a Lagrangian subspace of $(U, \sigma|_{U \times U})$. Then, it is easily seen that
    \begin{align*}
        G_{vec}^\sigma \supseteq (G_{vec})^{\sigma_{U}} \oplus (G_{vec})^{\perp} = (G_{vec}) \oplus (G_{vec})^{\perp}.
    \end{align*}
    Here, $\sigma_{U}$ denotes the symplectic complement in the symplectic vector space $U$ and $\perp$ denotes the orthogonal complement in the complex inner product space $\mathbb C^d$. Simple dimension counting shows that
    \begin{align*}
        \dim(G_{vec}) + \dim(iG_{vec}) + \dim(G_{vec}^{\perp}) = 2d,
    \end{align*}
    with the dimension meaning the real dimension. By a standard fact about dimensions on symplectic complements of subspaces, we therefore obtain:
    \begin{align*}
        G_{vec}^\sigma = (G_{vec})^{\sigma_{U}} \oplus (G_{vec})^{\sigma}.
    \end{align*}
    Similarly, one obtains:
    \begin{align*}
        G_{disc}^\sigma \supseteq G_{disc}^{\perp}.
    \end{align*}
    Indeed, $G_{disc}^{\perp}$ is the vector part of $G_{disc}^\sigma$. We have that
    \begin{align*}
       G_{vec} \oplus G_{disc} = G = G^\sigma = G_{disc}^\sigma \cap G_{vec}^\sigma.
    \end{align*}
    Hence, the vector parts of $G$ and $G_{disc}^\sigma \cap G_{vec}^\sigma$ have to agree. Since the vector part of $G_{disc}^\sigma \cap G_{vec}^\sigma$ is of course the intersection of the vector parts of $G_{disc}^\sigma$ and $G_{vec}^\sigma$, we arrive at $G_{vec} \subseteq G_{disc}^\perp$. We have therefore proven that $G_{disc} \perp U$. Let $V = U^\perp$, then $G_{disc} \subset V$. It is straightforward to verify that $\sigma|_{V\times V}$ is nondegenerate, such that $(\mathbb C^d, \sigma) = (U, \sigma|_{U \times U}) \oplus (V, \sigma|_{V \times V})$ as symplectic vector spaces. All that is left to be proven is that $G_{disc}$ is a Lagrangian subgroup of $V$. But this follows immediately from the fact that $\sigma(z, w) = \sigma(z, P_V w)$ for every $z \in G_{disc}$ and $w \in \mathbb C^d$, where $P_V$ is the orthogonal projection from $\mathbb C^d$ to $V$.
    \end{proof}
    Since any Lagrangian subspace of a $2m$-dimensional symplectic vector space is well-known to be symplectomorphic to a subspace of the form $\{ 0\}^{m} \oplus \mathbb R^{m}$, we are left with understanding the discrete part of $G$.
    \begin{lemma}\label{lem:gramschmidt}
        Let $(V, \sigma_0)$ be a finite dimensional symplectic vector space of dimension $2d$ and $G = G^{\sigma_0}$ a discrete Lagrangian subgroup. Then, there are 2-dimensional subspaces $V_1, \dots, V_d$ of $V$ such that $(V_j, \sigma_0|_{V_j \times V_j})$ is symplectic and $(G \cap V_j)$ is a Lagrangian subgroup of $(V_j, \sigma_0|_{V_j \times V_j})$ for every $j = 1, \dots, d$. Further, it holds true that $(V, \sigma_0) = (V_1, \sigma_0|_{V_1 \times V_1})  \oplus \dots \oplus (V_d, \sigma_0|_{V_d \times V_d})$, $G = (G \cap V_1) \times \dots \times (G \cap V_d)$.
    \end{lemma}
    \begin{proof}
         We first claim that there are $z_1, z_2 \in G$ such that $\sigma_0(z_1, z_2) = 2\pi$, which can be seen as follows: Let $0 \neq z \in G$. The condition $G = G^{\sigma_0}$ implies, in suggestive notation, that $\sigma_0(z, G) \subset 2\pi \mathbb Z$. If it were $\sigma_0(z, G) = \{ 0\}$, then we would have $\sigma_0(tz, G) = \{ 0\}$ for every $t \in \mathbb R$ such that $G$ would contain a real vector subspace, which we ruled out before. Therefore, $\sigma_0(z, G)$ has to be a subgroup of $2\pi \mathbb Z$, i.e.\ $\sigma_0(z, G) = 2\pi k \mathbb Z$ for some $k \in \mathbb N$. This shows $\sigma_0(z/k, G) = 2\pi \mathbb Z$, in particular $z_1 = z/k \in G$. Hence, there exists $z_2 \in G$ such that $\sigma_0(z_1, z_2) = 2\pi$.

We now consider the subspace $U = \operatorname{span}_{\mathbb R}\{ z_1, z_2\}$. Set $P_U(w) = \frac{\sigma_0(z_1, w)}{2\pi} z_2 - \frac{\sigma_0(z_2, w)}{2\pi}z_1$. Then, one easily sees that $P_U(G) \subset G$. Further, $P_U$ is a projection onto $U$. Let $P_W = I - P_U$ such that $P_W$ is a projection onto some subspace $W$ (the range of $P_W$). Further, $\mathbb R^{2d} = U \oplus W$ as the direct sum of vector spaces. We have $P_U(w) = w$ if and only if $w \in U$ and $P_U(w) = 0$ if and only if $w \in W$. As a consequence, each $w \in \mathbb R^{2d}$ can be decomposed as $w = P_U(w) + P_W(w)$, and hence $G = (G \cap U) \times (G \cap W)$. Now, it is not hard to verify that $\sigma_0(z, P_U(w)) = \sigma_0(P_U(z), w)$ for all $z, w \in V$. This implies:
\begin{align*}
    \sigma_0(z, w) &= \sigma_0(z, P_U(w)) + \sigma_0(z, P_W(w)) = \sigma_0(z, P_U^2(w)) + \sigma_0(z, P_W^2(w))\\
    &= \sigma_0(P_U(z), P_U(w)) + \sigma_0(P_W(z), P_W(w)).
\end{align*}
Therefore, we have the following direct sum of symplectic spaces:
\begin{align*}
    (V, \sigma_0) = (U, \sigma_0|_{U \times U}) \times (W, \sigma_0|_{W \times W}).
\end{align*}
Further, $G = (G \cap U) \times (G \cap W)$. Now, the last thing we have to prove is that $G \cap U$ and $G \cap W$ are Lagrangian in $U$ and $W$, respectively. Both go along analogous lines, so we only show the first. Let $z \in U$. Then, we have for $w \in G$:
\begin{align*}
    \sigma_0(z, w) = \sigma_0(z, P_U(w)).
\end{align*}
This directly implies that $G \cap U$ is Lagrangian in $U$.
    \end{proof}

    In the following last lemma, we again write $(\mathbb C, \sigma)$ for the symplectic vector space with $\mathbb C \cong \mathbb R^2$, $\sigma(z, w) = 2\im(z \overline{w})$.
\begin{lemma}\label{lemm:Lagrangian discrete subgroups d=1}
Let $(V, \sigma_0)$ be a real symplectic vector space of dimension 2. Further, let $G \subset V$ be a closed, discrete subgroup such that $G = G^{\sigma_0}$. Then, there exists linear symplectic map $S$, mapping from $(V, \sigma_0)$ to $(\mathbb R^2, \sigma)$ such that $S(G) = \sqrt{\pi} \mathbb Z \oplus \sqrt{\pi} \mathbb Z$.
\end{lemma}
\begin{proof}
    It is clear that $G$ is of rank 2: If we had $\operatorname{rank}(G) = 1$, then $G = x \mathbb Z$ for some $x \in V$ such that $tx \in G^\sigma$ for every $t \in \mathbb R$, i.e.\ $G$ contains the subspace $\mathbb R x$, which contradicts $G$ being discrete.

    Hence, we may assume that $G$ is generated by $f, g \in V$. Since $f, g$ generate $V$ as a vector space and $\sigma_0$ is non-degenerate, we know that $\sigma_0(f, g) \neq 0$. Since $f, g \in G$, we also have $\sigma_0(f, g) = 2k \pi$, some $k \in \mathbb Z \setminus \{ 0 \}$. Now, we can replace $g$ by $g/k$, as $\sigma_0(f, g/k) = 2\pi$, hence $g/k \in G^{\sigma_0} = G$. Note that even after this rescaling, $f$ and $g$ still generate $G$.  Now, we can set $P_1(y) = \frac{\sigma_0(f, y)}{2\pi} g$ and $P_2 = - \frac{\sigma_0(g, y)}{2\pi} f$. Obviously, both $P_1$ and $P_2$ are linear with $P_1^2 = P_1$, $P_2^2 = P_2$. Direct verification shows that $A(x) = P_1(x) + P_2(x)$ is a symplectic map, i.e.
    \begin{align*}
        \sigma_0(A(x), A(y)) = \sigma_0(x, y), \quad x, y \in V.
    \end{align*}
    Finally, we can define the map $S_0$ by $f \mapsto \sqrt{\pi}\cdot i$ and $g \mapsto \sqrt{\pi}\cdot 1$, where $1$ and $i$ can of course be understood as the standard basis elements $\begin{pmatrix}
        1 \\ 0
    \end{pmatrix}$ and $\begin{pmatrix}
        0 \\ 1
    \end{pmatrix}$ of $\mathbb R^2 \cong \mathbb C$. Since $\sigma(S_0 f, S_0 g) = 2\pi = \sigma_0(f, g)$, $S_0$ is also a symplectic map. The desired map is now $S = S_0 \circ A$.    
    \end{proof}

Combining the previous three lemmas, one obtains a proof of Theorem \ref{thm:structure_lagrangian_subgroups}.

\section*{Acknowledgements}
Robert Fulsche would like to thank Malte Gerhold, from whom he learned the proof of Lemma \ref{lem:gramschmidt} that we presented. The authors also acknowledge the anonymous referees' valuable comments on our work.

\bibliographystyle{abbrv}
\bibliography{main}

\begin{thebibliography}{10}

\bibitem{Arveson1982}
W.~Arveson.
\newblock {\em The harmonic analysis of automorphism groups}, volume~38 of {\em
  Proc. Sympos. Pure Math.}, page 199–269.
\newblock Amer. Math. Soc., 1982.

\bibitem{Arveson2003}
W.~Arveson.
\newblock {\em {Noncommutative Dynamics and E-Semigroups}}.
\newblock Springer Verlag, 2003.

\bibitem{Bauer_Isralowitz2012}
W.~Bauer and J.~Isralowitz.
\newblock {C}ompactness characterization of operators in the {T}oeplitz algebra
  of the {F}ock space {$F_\alpha^p$}.
\newblock {\em J. Funct. Anal.}, 263:1323--1355, 2012.

\bibitem{Bauer_Rodriguez2022}
W.~Bauer and M.~A. Rodriguez~Rodriguez.
\newblock {Commutative Toeplitz algebras and their Gelfand theory: old and new
  results}.
\newblock {\em Complex Anal. Oper. Theory}, 16, 2022.
\newblock 77.

\bibitem{Benedetto1975}
J.~J. Benedetto.
\newblock {\em Spectral Synthesis}.
\newblock Mathematische Leitf\"aden. Vieweg+Teubner Verlag, 1975.

\bibitem{Berge_Berge_Luef_Skrettingland2022}
E.~Berge, S.~M. Berge, F.~Luef, and E.~Skrettingland.
\newblock Affine quantum harmonic analysis.
\newblock {\em J. Funct. Anal.}, 282, 2022.
\newblock 109327.

\bibitem{Dawson_Olafsson_Quiroga2020}
M.~Dawson, G.~\'{O}lafsson, and R.~Quiroga-Barranco.
\newblock {Toeplitz Operators on the Domain $\{ Z \in M_{2 \times 2}(\mathbb
  C)|ZZ^\ast < I\}$ with $U(2) \times \mathbb T^2$-Invariant Symbols}.
\newblock In W.~Bauer, S.~Grudsky, and M.~Kaashoek, editors, {\em Operator
  Algebras, Toeplitz Operators and Related Topics}, volume 279 of {\em Operator
  Theory: Advances and Applications}. Birkh\"{a}user, 2020.

\bibitem{Dewage_Olafsson2022}
V.~Dewage and G.~\'{O}lafsson.
\newblock {Toeplitz Operators on the Fock Space with Quasi-Radial Symbols}.
\newblock {\em Complex Anal. Oper. Theory}, 16, 2022.
\newblock 61.

\bibitem{Esmeral_Maximenko2016}
K.~Esmeral and E.~A. Maximenko.
\newblock Radial {T}oeplitz operators on the {F}ock space and
  square-root-slowly oscillating sequences.
\newblock {\em Complex Anal. Oper. Theory}, 10:1655–1677, 2016.

\bibitem{Esmeral_Vasilevski2016}
K.~Esmeral and N.~Vasilevski.
\newblock {C*-algebra generated by horizontal Toeplitz operators on the Fock
  space}.
\newblock {\em Bol. Soc. Mat. Mex.}, 22:567–582, 2016.

\bibitem{Feichtinger_Kozek1998}
H.~G. Feichtinger and W.~Kozek.
\newblock {Quantization of TF lattice-invariant operators on elementary LCA
  groups}.
\newblock In H.~G. Feichtinger and T.~Strohmer, editors, {\em Gabor Analysis
  and Algorithms}, page 233–266. Birkh"auser, 1998.

\bibitem{Fulsche2020}
R.~Fulsche.
\newblock {Correspondence theory on $p$-Fock spaces with applications to
  Toeplitz algebras}.
\newblock {\em J. Funct. Anal.}, 279:108661, 2020.

\bibitem{Fulsche2022}
R.~Fulsche.
\newblock Toeplitz operators on non-reflexive {F}ock spaces.
\newblock preprint available on arXiv:2202.11440, 2022.

\bibitem{Hannabuss1979}
K.~Hannabuss.
\newblock {Representations of Nilpotent Locally Compact Groups}.
\newblock {\em J. Funct. Anal.}, 34:146–165, 1979.

\bibitem{Hernandez}
M.~Hern\'{a}ndez-Marroquin, A.~S\'{a}nchez-Nungaray, J.~L. Arroyo-Neri, and
  C.~Gonz\'{a}lez-Flores.
\newblock {Toeplitz Operators with Isotropic Invariant Symbols on the Fock
  Space}.
\newblock {\em Bull. Iranian Math. Soc.}, 49, 2023.

\bibitem{Hewitt_Ross2}
E.~Hewitt and K.~A. Ross.
\newblock {\em {Abstract Harmonic Analysis II}}, volume 152 of {\em Grundlehren
  der mathematischen Wissenschaften}.
\newblock Springer Verlag, 1970.

\bibitem{Hewitt_Ross1}
H.~Hewitt and K.~A. Ross.
\newblock {\em Abstract Harmonic Analysis. Structure of topological groups.
  Integration theory Volume 1}, volume 115 of {\em Grundlehren der
  mathematischen Wissenschaften}.
\newblock Springer Verlag, 2nd edition, 1994.

\bibitem{keyl_kiukas_werner16}
M.~Keyl, J.~Kiukas, and R.~Werner.
\newblock Schwartz operators.
\newblock {\em Rev. Math. Phys.}, 28, 2016.

\bibitem{Luef_Skrettingland2021}
F.~Luef and E.~Skrettingland.
\newblock A {W}iener {T}auberian theorem for operators and functions.
\newblock {\em J. Funct. Anal.}, 280, 2021.
\newblock 108883.

\bibitem{Mumford_Nori_Norman1991}
D.~Mumford, M.~Nori, and P.~Norman.
\newblock {\em {Tata Lectures on Theta III}}.
\newblock Birkh\"auser, 1991.

\bibitem{reiter71}
H.~Reiter.
\newblock {\em $L^1$-Algebras and Segal Algebras}, volume 231 of {\em Lecture
  Notes in Mathematics}.
\newblock Springer Verlag, 1971.

\bibitem{Rieffel1988}
M.~Rieffel.
\newblock {Projective Modules over Higher-Dimensional Non-Commutative Tori}.
\newblock {\em Canad. J. Math.}, 40:257–338, 1988.

\bibitem{Rozenblum_Vasilevski2022}
G.~Rozenblum and N.~Vasilevski.
\newblock {Commutative Algebras of Toeplitz Operators on the Bergman Space
  Revisited: Spectral Theorem Approach}.
\newblock {\em Integr. Equ. Oper. Theory}, 94:27, 2022.

\bibitem{Vasilevski2008}
N.~Vasilevski.
\newblock {\em {Commutative algebras of Toeplitz operators on the Bergman
  space}}.
\newblock Operator Theory: Advances and Applications. Birkh\"auser Verlag,
  2008.

\bibitem{werner84}
R.~Werner.
\newblock {Quantum Harmonic Analysis on Phase Space}.
\newblock {\em J. Math. Phys.}, 25:1404–1411, 1984.

\bibitem{Xia2015}
J.~Xia.
\newblock {Localization and the Toeplitz algebra on the Bergman space}.
\newblock {\em J. Funct. Anal.}, 269:781--814, 2015.

\bibitem{Zhu2012}
K.~Zhu.
\newblock {\em Analysis on {F}ock {S}paces}, volume 263 of {\em Graduate Texts
  in Mathematics}.
\newblock Springer US, New York, 2012.

\end{thebibliography}

\noindent
Robert Fulsche\\
\href{fulsche@math.uni-hannover.de}{\Letter fulsche@math.uni-hannover.de}
\\ ~\newline
Miguel Angel Rodriguez Rodriguez\\
\href{rodriguez@math.uni-hannover.de}{\Letter rodriguez@math.uni-hannover.de}
\\
~ \newline
Both authors:\\
\noindent
Institut f\"{u}r Analysis\\
Leibniz Universit\"at Hannover\\
Welfengarten 1\\
30167 Hannover\\
GERMANY

\end{document}